\documentclass[12pt,a4paper]{article}

\usepackage{amssymb}
\usepackage{amsfonts}
\usepackage{amsmath}
\usepackage{amsthm}
\usepackage{enumerate}

\usepackage{hyperref}

\usepackage[top=1 in, bottom=1 in, left=1 in, right=1 in]{geometry}

\numberwithin{equation}{section} 

\theoremstyle{definition}

\newtheorem{theorem}{Theorem}[section]

\newtheorem{corollary}[theorem]{Corollary}

\newtheorem{definition}[theorem]{Definition}
\newtheorem{example}[theorem]{Example}

\newtheorem{lemma}[theorem]{Lemma}

\newtheorem{proposition}[theorem]{Proposition}
\newtheorem{remark}[theorem]{Remark}

\begin{document}

\title{Schwarz type lemmas for pseudo-Hermitian manifolds}
\author{Yuxin Dong\footnote{Supported by NSFC grant No. 11771087, and LMNS, Fudan}, \ Yibin Ren\footnote{Supported by NSFC grant No. 11801517}, \ Weike Yu}

\date{}
\maketitle

\begin{abstract}
In this paper, we consider some generalized holomorphic maps between
pseudo-Hermitian manifolds. These maps include the \emph{CR} maps and the
transversally holomorphic maps. In terms of some sub-Laplacian or Hessian
type Bochner formulas, and comparison theorems in the pseudo-Hermitian
version, we are able to establish several Schwarz type results for both the 
\emph{CR} maps and the transversally holomorphic maps between
pseudo-Hermitian manifolds. Finally, we also discuss the \emph{CR} hyperbolicity
problem for pseudo-Hermitian manifolds.
\end{abstract}

\section*{Introduction}

The Schwarz lemma, reformulated by Pick \cite{[Pi]}, says that every holomorphic
map from a unit disc $D$ of $\mathbb{C}$ into itself is distance-decreasing with respect to the Poincar\'{e}
distance. This lemma is at the heart of geometric function theory, and has
been generalized to holomorphic maps between higher dimensional complex
spaces (\cite{[Ch]}, \cite{[Lu]}, \cite{[Ya]}, \cite{[CCL]}, \cite{[Ro]}, etc.), and also to quasiconformal
harmonic maps, harmonic maps with bounded (generalized) dilatation (\cite{[GI]},
\cite{[GIP]}, \cite{[Sh]}, etc.) between Riemannian manifolds in various ways. Some
Schwarz type lemmas have also been established in \cite{[Ta]} for almost
holomorphic maps between almost Hermitian manifolds and in \cite{[CZ]} for some
generalized harmonic maps.

The present paper is devoted to the study of Schwarz type lemmas for \emph{CR%
} manifolds. We will only consider \emph{CR} manifolds of hypersurface type
which admit positive definite pseudo-Hermitian structures. These are the
so-called pseudo-Hermitian manifolds (see \S \ref{section2} for the detailed definition).
Let $(M^{2m+1},H,J,\theta )$ denote a pseudo-Hermitian manifold of dimension 
$2m+1$, where $(H,J)$ is a \emph{CR} structure of type $(m,1)$, and $\theta $
is a pseudo-Hermitian strucutre on $M$. We find that the pseudo-Hermitian
manifold carries a rich geometric structure, including an almost complex
structure $J$ on $H$, a positive definite Levi form $L_{\theta }$ on $H$
induced by $\theta $ and $J$, and a $1$-dimensional foliation $\mathcal{F}%
_{\xi }$ called the Reeb foliation on $M$. In terms of the structures $\{J,H,%
\mathcal{F}_{\xi }\}$, one may thus introduce two important classes of
generalized holomorphic maps between any two pseudo-Hermitian manifolds,
that is, the \emph{CR} maps and the transversally holomorphic maps (see \S \ref{section3}
for their definitions). These maps are the main objects of our study. Note
that the pair $(H,L_{\theta })$ is a $2$-step sub-Riemannian structure,
which induces a Carnot--Carath\'{e}odory distance $r_{CC}$ on $M$. Besides,
the pseudo-Hermitian structure $\theta $ also induces both a volume form $%
\theta \wedge (d\theta )^{m}$ and a "horizontal volume form" $(d\theta )^{m}$
on $M$. These geometric data provide us a basis to investigate the Schwarz
type problem on pseudo-Hermitian manifolds from either the metric or the
volume aspect.

In this paper, we will give several Schwarz type lemmas for both the \emph{CR} maps and the transversally holomorphic maps. For this purpose, we need to
derive some sub-Laplacian or Hessian type Bochner formulas for these maps,
and establish suitable comparison theorems in the pseudo-Hermitian version.
Using these Bochner formulas and comparison theorems, we are able to
establish some distance-volume-decreasing properties (up to a constant) for
both the \emph{CR} maps and the transversally holomorphic maps under some
curvature lower bound for the source manifold and some negative curvature
upper bound for the target manifold. Our results may be regarded as the
realization that "negative curvature in the horizontal distribution of the
target \emph{CR} manifold restricsts these generalized holomorphic maps". In
view of these Schwarz type lemma, we also introduce some pseudo-distances
for \emph{CR} manifolds and investigate the correspongding phenomenon of
hyperbolicity.

\section{Pseudo-Hermitian geometry} \label{section2}

In this section, we introduce some notions and notations in pseudo-Hermitian
geometry (cf. \cite{[DT]}, \cite{[BG]} for details).

\begin{definition} \label{definition2.1}
Let $M$ be a real $2m+1$ dimensional orientable $%
C^{\infty }$ manifold. A \emph{CR} structure on $M$ is a complex subbundle $%
T_{1,0}M$ of the complexified tangent bundle $TM\otimes C$, of complex rank $%
m$, satisfying
\begin{enumerate}[(i)]
  \item $T_{1,0}M\cap T_{0,1}M=\{0\}$, $T_{0,1}M=\overline{T_{1,0}M}$;
  \item $[\Gamma (T_{1,0}M),\Gamma (T_{1,0}M)]\subseteq \Gamma (T_{1,0}M)$.
\end{enumerate}
Then the pair $(M,T_{1,0}M)$ is called a \emph{CR} manifold.
\end{definition}

The complex subbundle $T_{1,0}M$ corresponds to a real subbundle of $TM$%
\begin{equation}
H= \operatorname{Re} \{T_{1,0}M\oplus T_{0,1}M\} \label{equation2.1}
\end{equation}%
which is called the Levi distribution. It carries a natural complex
structure $J$ defined by $J(X+\overline{X})=i(X-\overline{X})$.
Equivalently, the \emph{CR} structure may be described by the pair $(H,J)$
too.

Let $E$ be the conormal bundle of $H$ in $T^{\ast }M$, whose fiber at each
point $x\in M$ is given by 
\begin{equation*}
E_{x}=\{\omega \in T_{x}^{\ast }M\mid \ker \omega \supseteq H_{x}\}.
\end{equation*}%
Since the complex structure $J$ induces an orientation on $H$, it follows
that the real line bundle $E$ ($\simeq TM/H$) is orientable, and thus
trivial. Consequently there exist globally defined nowhere vanishing
sections $\theta \in \Gamma (E)$. ~Any such a section $\theta $ is referred
to as a \emph{pseudo-Hermitian structure} on $M$. The Levi form $L_{\theta }$
of a given pseudo-Hermitian structure $\theta $ is defined by%
\begin{equation}
L_{\theta }(X,Y)=d\theta (X,JY)  \label{equation2.2}
\end{equation}%
for any $X,Y\in H$. The integrability condition (ii) in Definition \ref{definition2.1}
implies that $L_{\theta }$ is $J$-invariant, and thus symmetric. Recall that
if $L_{\theta }$ is positive definite on $H$ for some $\theta $, then $%
(M,H,J)$ is said to be strictly pseudoconvex. In this case, $\theta $ is
called a positive pseudo-Hermitian structure. Henceforth we will assume that 
$(M,H,J)$ is a strictly pseudoconvex \emph{CR} manifold endowed with a
positive pseudo-Hermitian structure $\theta $. The quadruple $(M,H,J,\theta
) $ is referred to as a \emph{pseudo-Hermitian manifold}.

Since $L_{\theta }$ is positive definite, we have a sub-Riemannian structure 
$(H,L_{\theta })$ on $M$, and all sections of $H$ together with their Lie
brackets span $T_{x}M$ at each point $x$. Indeed $(M,H,L_{\theta })$ is a
step-$2$ sub-Riemannian manifold. We say that a Lipschitz curve $\gamma
:[0,l]\rightarrow M$ is horizontal if $\gamma ^{\prime }(t)\in H_{\gamma
(t)} $ a.e. in $[0,l]$. For any two points $p$, $q\in M$, by the theorem of
Chow-Rashevsky (\cite{[Cho]}, \cite{[Ra]}), there always exist such horizontal curves
joining $p$ and $q$. Consequently, we may define the so-called Carnot--Carath%
\'{e}odory distance: 
\begin{equation}
r_{CC}(p,q)=\inf \left\{\int_{0}^{l}\sqrt{L_{\theta }(\gamma ^{\prime
},\gamma ^{\prime })}dt \  \bigg| 
\begin{array}{c}
\gamma :[0,l]\rightarrow M\text{ is a horizontal
curve,} \\ 
\gamma (0)=p\text{, }\gamma (l)=q 
\end{array} \right\} \label{equation2.3}
\end{equation}
which induces to a metric space structure on $(M,H,L_{\theta })$.

For a pseudo-Hermitian manifold $(M,H,J,\theta )$, there is a unique
globally defined nowhere zero vector field $\xi $ on $M$ such that 
\begin{equation}
\theta (\xi )=1\text{, \ }d\theta (\xi ,\cdot )=0.  \label{equation2.4}
\end{equation}%
This vector field is called the \emph{Reeb vector field}, whose integral
curves forms an oriented one-dimensional foliation $\mathcal{F}_{\xi }$ on $ M $. Clearly $TM$ admits the following decomposition%
\begin{equation}
TM=H\oplus L  \label{equation2.5}
\end{equation}%
where $L$ is the trivial line bundle generated by $\xi $. One may extend the
complex structure $J$ to an endomorphism of $TM$ by requiring%
\begin{equation}
J\xi =0.  \label{equation2.6}
\end{equation}%
Let $\pi _{H}:TM\rightarrow H$ denote the natural projection morphism and
set 
\begin{equation}
G_{\theta }(X,Y)=L_{\theta }(\pi _{H}(X),\pi _{H}(Y))   \label{equation2.7}
\end{equation}%
for any $X,Y\in TM$. Then the Levi form $L_{\theta }$ can be extended to a
Riemannian metric on $M$ by 
\begin{equation}
g_{\theta }=G_{\theta }+\theta \otimes \theta , \label{equation2.8}
\end{equation}%
which is called the Webster metric. Thus we have a Riemannian distance
function $r$ of $g_{\theta }$, which is a useful auxiliary function for
studying geometric analysis on $M$. We find from \eqref{equation2.8} that \eqref{equation2.5} is
actually an orthogonal decomposition of $TM$ with respect to $g_{\theta }$.
The volume form of $g_{\theta }$ is given, up to a constant, by%
\begin{equation}
\Omega _{M}=\theta \wedge (d\theta )^{m}.  \label{equation2.9}
\end{equation}%
One may also introduce the "horizontal volume form" 
\begin{equation}
\Omega _{M}^{H}=\left( d\theta \right) ^{m}.  \label{equation2.10}
\end{equation}

On a pseudo-Hermtian manifold, there is a canonical linear connection
preserving both the \emph{CR} structure and the Webster metric.

\begin{theorem}[{\cite{[DT]}}] \label{theorem2.1}
  Let $(M,H,J,\theta )$ be a pseudo-Hermitian
  manifold with the Reeb vector field $\xi $ and the Webster metric $g_{\theta
  }$. Then there exists a unique linear connection $\nabla $ such that
  \begin{enumerate}[(i)]
    \item $\ \nabla _{X}\Gamma (H)\subset \Gamma (H)$ for any $X\in \Gamma (TM)$;
    \item $\ \nabla J=0$, $\nabla g_{\theta }=0$;
    \item The torsion $T_{\nabla }$ of $\nabla $ satisfies 
    \begin{equation*}
    T_{\nabla }(X,Y)=2d\theta (X,Y)\xi \text{ \ and \ }T_{\nabla }(\xi
    ,JX)+JT_{\nabla }(\xi ,X)=0
    \end{equation*}
    for any $X,Y\in H$.
  \end{enumerate}
\end{theorem}

The connection $\nabla $ in Theorem \ref{theorem2.1} is usuallly called the \emph{%
Tanaka-Webster connection.} The \emph{pseudo-Hermitian torsion} of $\nabla $%
, denoted by $\tau $, is a $TM$-valued 1-form defined by $\tau (X)=T_{\nabla
}(\xi ,X)$ for any $X\in TM$. It induces a trace-free symmetric tensor field 
$A$ given by%
\begin{equation}
A(X,Y)=g_{\theta }(\tau (X),Y). \label{equation2.11}
\end{equation}%
A pseudo-Hermitian manifold is said to be \emph{Sasakian} if $\tau =0$ (or
equivalently, $A=0$). Sasakian manifolds can be viewed as an odd-dimensional
analogue of K\"{a}hler manifolds (cf. \cite{[BG]}).

Suppose $(M^{2m+1},H,J,\theta )$ is a $2m+1$ dimensional pseudo-Hermitian
manifold with the Webster metric $g_{\theta }$ and Tananka-Webster
connection $\nabla $. Let $\eta _{1},\eta _{2},...,\eta _{m}$ be a unitary
frame field of $T_{1,0}M$ on an open domain of $M$, and let $\theta
^{1},\theta ^{2},...,\theta ^{m}$ be its coframe field. Then the "horizontal
component of $g_{\theta }$" may be expressed as%
\begin{equation}
G_{\theta }=\sum_{i=1}^{m}\theta ^{i}\theta ^{\overline{i}}. \label{equation2.12}
\end{equation}%
Note that (iii) of Theorem \ref{theorem2.1} implies that $\tau (T_{1,0}M)\subseteq
T_{0,1}M$. In terms of the frame fields, one may write%
\begin{align}
\tau =\sum_{i}\left( \tau ^{i}\eta _{i}+\tau ^{\overline{i}}\eta _{%
\overline{i}}\right) 
=\sum_{i, j} \left( A_{\overline{j}}^{i}\theta ^{\overline{j}}\eta _{i}+A_{j}^{%
\overline{i}}\theta ^{j}\eta _{\overline{i}}\right)  \label{equation2.13} 
\end{align}%
From \cite{[We]}, we have the following structure equations for the Tanaka-Webster
connection $\nabla $: 
\begin{equation}
\left\{ 
\begin{aligned}
d\theta \text{ } &= 2\sqrt{-1}\sum_{j}\theta ^{j}\wedge \theta ^{\overline{j}%
} \\ 
d\theta ^{i} & = -\sum_{j}\theta _{j}^{i}\wedge \theta ^{j}+\theta
\wedge \tau ^{i} \\ 
d\theta _{j}^{i} &= -\sum_{k}\theta _{k}^{i}\wedge \theta _{j}^{k}+\Pi
_{j}^{i}%
\end{aligned}%
\right. \label{equation2.14}
\end{equation}%
with%
\begin{equation}
\theta _{j}^{i}+\theta _{\overline{i}}^{\overline{j}}=0 \label{equation2.15}
\end{equation}%
and%
\begin{align}
\Pi _{j}^{i} =& \: 2\sqrt{-1}\theta ^{i}\wedge \tau ^{\overline{j}}-2\sqrt{-1}%
\tau ^{i}\wedge \theta ^{\overline{j}}+\sum_{k,l}R_{jk\overline{l}%
}^{i}\theta ^{k}\wedge \theta ^{\overline{l}} \label{equation2.16} \\
&+\sum_{k}\left( W_{jk}^{i}\theta ^{k}\wedge \theta -W_{j\overline{k}%
}^{i}\theta ^{\overline{k}}\wedge \theta \right)  \notag
\end{align}%
where $W_{jk}^{i}=A_{j,\overline{i}}^{\overline{k}}$, $W_{j\overline{k}%
}^{i}=A_{\overline{k},j}^{i}$, and $\{R_{jk\overline{l}}^{i}\}$ are
components of the curvature tensor of $\nabla $. Set $R_{i\overline{j}k%
\overline{l}}=R_{ik\overline{l}}^{j}$. Then we know that%
\begin{align}\label{equation2.17}
R_{i\overline{j}k\overline{l}} = -R_{\overline{j}ik\overline{l}}=-R_{i%
\overline{j}\overline{l}k} , \quad 
R_{i\overline{j}k\overline{l}} = R_{k\overline{j}i\overline{l}}=R_{k%
\overline{l}i\overline{j}}. 
\end{align}

Suppose $X=X^{i}\eta _{i}$ and $Y=Y^{j}\eta_{j}$ are two vectors in $%
T_{1,0}M$, then the \emph{pseudo-Hermitian bisectional curvature} determined
by $X$ and $Y$ is defined by%
\begin{equation}
\frac{\sum_{i,j,k,l}R_{i\overline{j}k\overline{l}}X^{i}X^{\overline{j}%
}Y^{k}Y^{\overline{l}}}{(\sum_{i}X^{i}X^{\overline{i}})(\sum_{j}Y^{j}Y^{%
\overline{j}})}.  \label{equation2.18}
\end{equation}%
If $X=Y$, the above quantity is called the \emph{pseudo-Hermitian sectional
curvature} in the direction $X$ (cf. \cite{[We]}). The \emph{pseudo-Hermitian Ricci
curvature} is defined as%
\begin{equation}
R_{i\overline{j}}=\sum_{k}R_{k\overline{k}i\overline{j}}  \label{equation2.19}
\end{equation}%
and thus the \emph{pseudo-Hermitian scalar curvature} is given by%
\begin{equation}
R=\sum_{i}R_{i\overline{i}}  \label{equation2.20}
\end{equation}

The simpliest \emph{CR} manifolds are Sasakian space forms, which are simply
connected Sasakian manifolds with constant pseudo-hermitian curvature. For
our purpose, let us recall the Sasakian space form with negative constant
pseudohermitian curvature as follows.

\begin{example}[(cf. {\cite{[BG]}} page 229)] \label{example2.1}
Let $B_{C}^{m}(1)$ be the complex
ball in $C^{m}$ with the Bergman metric $g_{0}$ of constant holomorphic
sectional curvature $-1$. One can scale $g_{0}$ to have a metric $g_{B}$
with contant holomorphic sectional sectional curvature $-k$ ($k>0$). Let $\omega_{B}$ be the K\"{a}hler form of $g_{B}$. Clearly there is a $1$-form $%
\alpha _{B}$ on $B_{C}^{m}(1)$ such that $\omega _{B}=d\alpha _{B}$, since $%
B_{C}^{m}(1)$ is simply connected. Set $\theta =dt+\pi ^{\ast }\alpha _{B}$
and $\xi =\frac{\partial }{\partial t}$, where $\pi :B_{C}^{m}(1)\times
R\rightarrow B_{C}^{m}(1)$ is the natural projection and $t$ is the
coordinate on $R$. Set $H=\ker \theta $. We define an almost complex
structure $J$ on $H$ to be the horizontal lift of the almost complex
structure $J_{B}$ on $B_{C}^{m}(1)$. Then $(B_{C}^{m}(1)\times R,H,J,\theta
) $ is a Sasakian space form with constant pseudohermitian curvature $-k$,
which will also be denoted by $D^{2m+1}(-k)$.
\end{example}

Analogous to the Laplace operator in Riemiannian geometry, there is a
degenerate elliptic operator in pseudo-Hermitian geometry, which is called
sub-Laplace operator. For a smooth function $u:(M,H,J,\theta )\rightarrow 
\mathbb{R}
$, let $\nabla du$ be the convariant derivative of the differential $du\in
\Gamma (T^{\ast }M)$ with respect to the Tanaka-Webster connection. Then the 
\emph{sub-Laplacian} of $u$ is defined by%
\begin{equation}
\bigtriangleup _{b}u=tr_{H}\left( \nabla du\right) =\sum_{i}\left( u_{i%
\overline{i}}+u_{\overline{i}i}\right)  \label{equation2.21}
\end{equation}%
where $u_{i\overline{i}}=\left( \nabla du\right) (\eta _{i},\eta _{\overline{%
i}})$.

Suppose $(N^{2n+1},\widetilde{H},\widetilde{J},\widetilde{\theta })$ is
another pseudo-Hermitian manifold with the Webster metric $g_{\widetilde{%
\theta }}$ and Tanaka-Webster connection $\widetilde{\nabla }$. Let $\{%
\widetilde{\xi },\widetilde{\eta }_{\alpha },\widetilde{\eta }_{\overline{%
\alpha }}\}$ be a local frame field on the $N$ with $\widetilde{\xi }$ the
Reeb vector field determined by $\widetilde{\theta }$ and $\{\widetilde{\eta 
}_{\alpha }\}$ the unitary frame field of $T_{1,0}N$. Let $\{\widetilde{%
\theta },\widetilde{\theta }^{\alpha },\widetilde{\theta }^{\overline{\alpha 
}}\}$ be its dual frame field. We will denote the corresponding geometric
data, such as, the connection $1$-forms, torsion and curvature, etc., on $N$
by the same notations as in $M$, but with \symbol{126}on them. Then similar
structure equations for $\widetilde{\nabla }$ are valid in $N$ too.

\section{Bochner formulas for \texorpdfstring{$(H,\protect\widetilde{H})$}{}-holomorphic maps} \label{section3}

In this section, we will derive Bochner formulas for some generalized
holomorphic maps between two pseudo-Hermitian manifolds.

\begin{definition}[{\cite{[Do]}}] \label{definition3.1}
We say that a map $f:(M,H,J,\theta
)\rightarrow (N,\widetilde{H},\widetilde{J},\widetilde{\theta })$ between
two pseduo-Hermitian manifolds is $(H,\widetilde{H})$-holomorphic if it
satisfies%
\begin{equation}
df_{H,\widetilde{H}}\circ J=\widetilde{J}\circ df_{H,\widetilde{H}} \label{equation3.1}
\end{equation}%
where $df_{H,\widetilde{H}}=\pi _{\widetilde{H}}\circ df\circ i_{H}$, $\pi _{%
\widetilde{H}}:TN\rightarrow \widetilde{H}$ is the natural projection
morphism and $i_{H}:H\rightarrow TM$ is the inclusion morphism.
\end{definition}

Let $f:(M^{2m+1},H,J,\theta )\rightarrow (N,\widetilde{H},\widetilde{J},%
\widetilde{\theta })$ be a $(H,\widetilde{H})$-holomorphic map between
pseudo-Hermitian manifolds. Then its differential can be expressed as%
\begin{equation}
df=\sum_{A, B}f_{B}^{\widetilde{A}}\theta ^{B}\otimes \widetilde{\eta }_{%
\widetilde{A}}  \label{equation3.2}
\end{equation}%
where $\theta ^{0}=\theta $, $\widetilde{\eta }_{0}=\widetilde{\xi }$. We
use the following convention on the ranges of indices in this paper: 
\begin{eqnarray*}
A,B,C &=&0,1,...,m,\overline{1},...,\overline{m};\text{ \ }i,j,k=1,...,m,%
\overline{i},\overline{j},\overline{k}=\overline{1},...,\overline{m}; \\
\widetilde{A},\widetilde{B},\widetilde{C} &=&0,1,...,n,\overline{1},...,%
\overline{n};\text{ \ }\alpha ,\beta ,\gamma =1,...,n,\overline{\alpha },%
\overline{\beta },\overline{\gamma }=\overline{1},...,\overline{n}.
\end{eqnarray*}%
Clearly the condition \eqref{equation3.1} in Definition \ref{definition3.1} is equivalent to%
\begin{equation}
f_{\overline{i}}^{\alpha }=f_{i}^{\overline{\alpha }}=0.  \label{equation3.3}
\end{equation}%
From \eqref{equation3.2} and \eqref{equation3.3}, we have%
\begin{align}
f^{\ast }\widetilde{\theta } & =f_{0}^{0}\theta +\sum_{j}(f_{j}^{0}\theta
^{j}+f_{\overline{j}}^{0}\theta ^{\overline{j}}) \label{equation3.4} \\  
f^{\ast }\widetilde{\theta }^{\alpha } & =f_{0}^{\alpha }\theta
+\sum_{j}f_{j}^{\alpha }\theta ^{j}.  \label{equation3.5}
\end{align}

By taking the exterior derivative of \eqref{equation3.4} and using the structure equations
in $M$ and $N$, we obtain%
\begin{align}
0 = & Df_{0}^{0}\wedge \theta +\sum_{j}(Df_{j}^{0}\wedge \theta ^{j}+Df_{\overline{%
j}}^{0}\wedge \theta ^{\overline{j}})+\sum_{i,j}(f_{i}^{0}A_{\overline{j}%
}^{i}\theta \wedge \theta ^{\overline{j}}+f_{\overline{i}}^{0}A_{j}^{%
\overline{i}}\theta \wedge \theta ^{j}) \nonumber \\ 
& +2\sqrt{-1}f_{0}^{0}\sum_{j}\theta ^{j}\wedge \theta ^{\overline{j}}-2\sqrt{
-1}\sum_{\alpha }f^{\ast }\widetilde{\theta }^{\alpha }\wedge f^{\ast }%
\widetilde{\theta }^{\overline{\alpha }}  \label{equation3.6}
\end{align}%
where 
\begin{align*}
Df_{0}^{0} &= df_{0}^{0}=f_{00}^{0}\theta +\sum_{k}(f_{0k}^{0}\theta
^{k}+f_{0\overline{k}}^{0}\theta ^{\overline{k}}) \\
Df_{j}^{0} &= df_{j}^{0}-\sum_{k}f_{k}^{0}\theta _{j}^{k}=f_{j0}^{0}\theta
+\sum_{k}(f_{jk}^{0}\theta ^{k}+f_{j\overline{k}}^{0}\theta ^{\overline{k}})
\\
Df_{\overline{j}}^{0} &= df_{\overline{j}}^{0}-\sum_{k}f_{\overline{k}%
}^{0}\theta _{\overline{j}}^{\overline{k}}=f_{\overline{j}0}^{0}\theta
+\sum_{k}(f_{\overline{j}k}^{0}\theta ^{k}+f_{\overline{j}\overline{k}%
}^{0}\theta ^{\overline{k}}).
\end{align*}%
Then \eqref{equation3.6} gives%
\begin{align}
\begin{aligned}
& f_{jl}^{0} = f_{lj}^{0} \\ 
& f_{0l}^{0}-f_{l0}^{0}-\sum_{i}f_{\overline{i}}^{0}A_{j}^{\overline{i}} =2
\sqrt{-1}\sum_{\alpha }f_{l}^{\alpha }f_{0}^{\overline{\alpha }} \\ 
& f_{j\overline{l}}^{0}-f_{\overline{l}j}^{0}-2\sqrt{-1}f_{0}^{0}\delta
_{j}^{l} =-2\sqrt{-1}\sum_{\alpha }f_{j}^{\alpha }f_{\overline{l}}^{\overline{ \alpha }} 
\end{aligned}
\label{equation3.7}
\end{align}

By similar computations, we get from \eqref{equation3.5} that%
\begin{equation}
Df_{0}^{\alpha }\wedge \theta +\sum_{j}Df_{j}^{\alpha }\wedge \theta ^{j}+2%
\sqrt{-1}\sum_{j}f_{0}^{\alpha }\theta ^{j}\wedge \theta ^{\overline{j}%
}+\sum_{j,k}f_{j}^{\alpha }A_{\overline{k}}^{j}\theta \wedge \theta ^{%
\overline{k}}=\sum_{\beta }\widetilde{A}_{\overline{\beta }}^{\alpha
}f^{\ast }\widetilde{\theta }\wedge f^{\ast }\widetilde{\theta }^{\overline{%
\beta }}  \label{equation3.8}
\end{equation}%
where%
\begin{align}
Df_{0}^{\alpha } & = df_{0}^{\alpha }+\sum_{\beta }f_{0}^{\beta }\widetilde{%
\theta }_{\beta }^{\alpha }=f_{00}^{\alpha }\theta +\sum_{l}(f_{0l}^{\alpha
}\theta ^{l}+f_{0\overline{l}}^{\alpha }\theta ^{\overline{l}})  \label{equation3.9} \\
Df_{j}^{\alpha } & = df_{j}^{\alpha }-\sum_{k}f_{k}^{\alpha }\theta
_{j}^{k}+\sum_{\beta }f_{j}^{\beta }\widetilde{\theta }_{\beta }^{\alpha
}=f_{j0}^{\alpha }\theta +\sum_{l}(f_{jl}^{\alpha }\theta ^{l}+f_{j\overline{%
l}}^{\alpha }\theta ^{\overline{l}}).  \label{equation3.10}
\end{align}%
Using \eqref{equation3.9} and \eqref{equation3.10}, we deduce from \eqref{equation3.8} that%
\begin{align}
\begin{aligned}
& f_{kl}^{\alpha } = f_{lk}^{\alpha } \\ 
& f_{0j}^{\alpha }-f_{j0}^{\alpha } = \sum_{\beta }\widetilde{A}_{\overline{%
\beta }}^{\alpha }f_{j}^{0}f_{0}^{\overline{\beta }} \\ 
& f_{0\overline{j}}^{\alpha }-\sum_{k}f_{k}^{\alpha }A_{\overline{j}
}^{k} = \sum_{\beta }\widetilde{A}_{\bar{\beta} }^{\alpha }\left( f_{\overline{j}
}^{0}f_{0}^{\bar{\beta} }-f_{\overline{j}}^{\overline{\beta }}f_{0}^{0}\right) \\ 
& f_{j\overline{l}}^{\alpha }-2\sqrt{-1}f_{0}^{\alpha }\delta
_{j}^{l} = -\sum_{\beta }\widetilde{A}_{\overline{\beta }}^{\alpha
}f_{j}^{0}f_{\overline{l}}^{\overline{\beta }} \\ 
& \sum_{\beta }\widetilde{A}_{\overline{\beta }}^{\alpha }(f_{\overline{l}%
}^{0}f_{\overline{j}}^{\overline{\beta }}-f_{\overline{j}}^{0}f_{\overline{l}%
}^{\overline{\beta }}) = 0.%
\end{aligned}%
\label{equation3.11}
\end{align}

Applying the exterior derivative to \eqref{equation3.10} and using the structure equations
again, we obtain%
\begin{equation}
\begin{aligned}
Df_{j0}^{\alpha }\wedge \theta +\sum_{l}(Df_{jl}^{\alpha }\wedge \theta
^{l}+Df_{j\overline{l}}^{\alpha }\wedge \theta ^{\overline{l}}+2\sqrt{-1}%
f_{j0}^{\alpha }\theta ^{l}\wedge \theta ^{\overline{l}})& \\ 
+\sum_{k,l}(f_{jk}^{\alpha }A_{\overline{l}}^{k}\theta \wedge \theta ^{%
\overline{l}}+f_{j\overline{k}}^{\alpha }A_{l}^{\overline{k}}\theta \wedge
\theta ^{l}) & = -f_{i}^{\alpha }\Pi _{j}^{i}+f_{j}^{\beta }\widetilde{\Pi }%
_{\beta }^{\alpha }%
\end{aligned}%
\label{equation3.12}
\end{equation}%
where%
\begin{equation}
\begin{aligned}
Df_{j0}^{\alpha } & = df_{j0}^{\alpha }-\sum_{k}f_{k0}^{\alpha }\theta
_{j}^{k}+\sum_{\beta }f_{j0}^{\beta }\widetilde{\theta }_{\beta }^{\alpha
}=f_{j00}^{\alpha }\theta +\sum_{l}(f_{j0l}^{\alpha }\theta ^{l}+f_{j0%
\overline{l}}^{\alpha }\theta ^{\overline{l}}) \\ 
Df_{jl}^{\alpha } & = df_{jl}^{\alpha }-\sum_{k}f_{kl}^{\alpha }\theta
_{j}^{k}-\sum_{k}f_{jk}^{\alpha }\theta _{l}^{k}+\sum_{\beta }f_{jl}^{\beta }%
\widetilde{\theta }_{\beta }^{\alpha }=f_{jl0}^{\alpha }\theta
+\sum_{k}(f_{jlk}^{\alpha }\theta ^{k}+f_{jl\overline{k}}^{\alpha }\theta ^{%
\overline{k}}) \\ 
Df_{j\overline{l}}^{\alpha } & = df_{j\overline{l}}^{\alpha }-\sum_{k}f_{k%
\overline{l}}^{\alpha }\theta _{j}^{k}-\sum_{k}f_{j\overline{k}}^{\alpha
}\theta _{\overline{l}}^{\overline{k}}+\sum_{\beta }f_{j\overline{l}}^{\beta
}\widetilde{\theta }_{\beta }^{\alpha }=f_{j\overline{l}0}^{\alpha }\theta
+\sum_{k}(f_{j\overline{l}k}^{\alpha }\theta ^{k}+f_{j\overline{l}\overline{k%
}}^{\alpha }\theta ^{\overline{k}}).%
\end{aligned}%
\label{equation3.13}
\end{equation}%
From \eqref{equation3.12}, we collect terms corresponding $\theta ^{k}\wedge \theta ^{%
\overline{l}}$ to get%
\begin{align}
f_{j\overline{l}k}^{\alpha }-f_{jk\overline{l}}^{\alpha } &= -2\sqrt{-1}%
f_{j0}^{\alpha }\delta _{k}^{l}-\sum_{i}f_{i}^{\alpha }R_{jk\overline{l}%
}^{i}+\sum_{\beta ,\gamma ,\delta }f_{j}^{\beta }f_{k}^{\gamma }f_{\overline{%
l}}^{\overline{\delta }}\widetilde{R}_{\beta \gamma \overline{\delta }%
}^{\alpha }  \label{equation3.14} \\
& \quad +\sum_{\beta ,\gamma }(f_{j}^{\beta }f_{k}^{\gamma }f_{\overline{l}}^{0}%
\widetilde{W}_{\beta \gamma }^{\alpha }+f_{j}^{\beta }f_{\overline{l}}^{%
\overline{\gamma }}f_{k}^{0}\widetilde{W}_{\beta \overline{\gamma }}^{\alpha
}).  \notag
\end{align}

The horizontal energy density of the $(H,\widetilde{H})$-holomorphic map $f$
is given by 
\begin{equation}
e_{H,\widetilde{H}}=\frac{1}{2}\mid df_{H,\widetilde{H}}\mid
^{2}=\sum_{\alpha ,j}f_{j}^{\alpha }f_{\overline{j}}^{\overline{\alpha }}. 
\label{equation3.15}
\end{equation}%
We first derive the Bochner formula for $e_{H,\widetilde{H}}$ as follows.
From \eqref{equation3.15}, we have%
\begin{equation*}
(e_{H,\widetilde{H}})_{k}=\sum_{\alpha ,j}\left( f_{jk}^{\alpha }f_{%
\overline{j}}^{\overline{\alpha }}+f_{j}^{\alpha }f_{\overline{j}k}^{%
\overline{\alpha }}\right) ,
\end{equation*}%
and thus%
\begin{equation}
\left( e_{H,\widetilde{H}}\right) _{k\overline{k}}=\sum_{\alpha ,j}\left(
\mid f_{jk}^{\alpha }\mid ^{2}+\mid f_{j\overline{k}}^{\alpha }\mid ^{2}+f_{%
\overline{j}}^{\overline{\alpha }}f_{jk\overline{k}}^{\alpha }+f_{j}^{\alpha
}f_{\overline{j}k\overline{k}}^{\overline{\alpha }}\right) .  \label{equation3.16}
\end{equation}%
In terms of \eqref{equation2.21} and \eqref{equation3.16}, we get%
\begin{equation}
\bigtriangleup _{b}e_{H,\widetilde{H}}=2\sum_{\alpha ,j,k}\left( \mid
f_{jk}^{\alpha }\mid ^{2}+\mid f_{j\overline{k}}^{\alpha }\mid ^{2}\right)
+\sum_{\alpha ,j,k}\left( f_{\overline{j}}^{\overline{\alpha }}f_{jk%
\overline{k}}^{\alpha }+f_{j}^{\alpha }f_{\overline{j}k\overline{k}}^{%
\overline{\alpha }}+f_{j}^{\alpha }f_{\overline{j}\overline{k}k}^{\overline{%
\alpha }}+f_{\overline{j}}^{\overline{\alpha }}f_{j\overline{k}k}^{\alpha
}\right) .  \label{equation3.17}
\end{equation}%
Using \eqref{equation3.11} and \eqref{equation3.14}, we perform the following computations%
\begin{equation} 
\begin{aligned}
f_{jk\overline{k}}^{\alpha }=f_{kj\overline{k}}^{\alpha } & = f_{k\overline{k}%
j}^{\alpha }+2\sqrt{-1}f_{k0}^{\alpha }\delta _{j}^{k}+f_{i}^{\alpha
}R_{kjk}^{i}-f_{k}^{\beta }f_{j}^{\gamma }f_{\overline{k}}^{\overline{\delta 
}}\widetilde{R}_{\beta \gamma \overline{\delta }}^{\alpha } \\ 
& \quad -f_{k}^{\beta }f_{j}^{\gamma }f_{\overline{k}}^{0}\widetilde{W}_{\beta
\gamma }^{\alpha }-f_{k}^{\beta }f_{\overline{k}}^{\overline{\gamma }%
}f_{j}^{0}\widetilde{W}_{\beta \overline{\gamma }}^{\alpha }%
\end{aligned}%
\label{equation3.18}
\end{equation}%
and%
\begin{equation}
f_{j\overline{k}k}^{\alpha }=\left( 2\sqrt{-1}f_{0}^{\alpha }\delta
_{j}^{k}-\sum_{\beta }\widetilde{A}_{\overline{\beta }}^{\alpha }f_{j}^{0}f_{%
\overline{k}}^{\overline{\beta }}\right) _{k}=2\sqrt{-1}f_{0k}^{\alpha
}\delta _{j}^{k}-\left( \sum_{\beta }\widetilde{A}_{\overline{\beta }%
}^{\alpha }f_{j}^{0}f_{\overline{k}}^{\overline{\beta }}\right) _{k} 
\label{equation3.19}
\end{equation}%
Clearly the conjugates of \eqref{equation3.18} and \eqref{equation3.19} yields the expressions of $f_{%
\overline{j}\overline{k}k}^{\overline{\alpha }}$ and $f_{\overline{j}k%
\overline{k}}^{\overline{\alpha }}$. From \eqref{equation3.11}, \eqref{equation3.18} and \eqref{equation3.19}, it follows that 
\begin{equation}
\begin{aligned}
\bigtriangleup _{b}e_{H,\widetilde{H}} & = 2\sum_{\alpha ,j,k}\left( \mid
f_{jk}^{\alpha }\mid ^{2}+\mid f_{j\overline{k}}^{\alpha }\mid ^{2}\right)
+2(m+2)\sqrt{-1}\sum_{j,\alpha }\left( f_{\overline{j}}^{\overline{\alpha }%
}f_{0j}^{\alpha }-f_{j}^{\alpha }f_{0\overline{j}}^{\overline{\alpha }%
}\right) +2\sum_{i,j,k,\alpha }f_{i}^{\alpha }f_{\overline{j}}^{\overline{\alpha }}R_{k\overline{i}j\overline{k}} \\ 
& \quad -2\sum_{j,k,\alpha ,\beta ,\gamma ,\delta}f_{k}^{\alpha }f_{\overline{j}}^{\overline{\beta }}f_{j}^{\gamma }f_{\overline{k}}^{\overline{\delta }}\widetilde{R}_{\alpha \overline{\beta }%
\gamma \overline{\delta }} +2\sqrt{-1}\sum_{j,\alpha ,\beta }\left( \widetilde{A}_{\beta }^{\overline{%
\alpha }}f_{\overline{j}}^{0}f_{0}^{\beta }f_{j}^{\alpha }-\widetilde{A}_{%
\overline{\beta }}^{\alpha }f_{j}^{0}f_{0}^{\overline{\beta }}f_{\overline{j}%
}^{\overline{\alpha }}\right) \\ 
& \quad -\sum_{j,k,\alpha ,\beta ,\gamma }\left( f_{\overline{j}}^{\overline{\alpha }%
}f_{k}^{\beta }f_{j}^{\gamma }f_{\overline{k}}^{0}\widetilde{W}_{\beta
\gamma }^{\alpha }+f_{\overline{j}}^{\overline{\alpha }}f_{k}^{\beta }f_{%
\overline{k}}^{\overline{\gamma }}f_{j}^{0}\widetilde{W}_{\beta \overline{%
\gamma }}^{\alpha }+f_{j}^{\alpha }f_{\overline{k}}^{\overline{\beta }}f_{%
\overline{j}}^{\overline{\gamma }}f_{k}^{0}\widetilde{W}_{\overline{\beta }%
\overline{\gamma }}^{\overline{\alpha }}+f_{j}^{\alpha }f_{\overline{k}}^{%
\overline{\beta }}f_{k}^{\gamma }f_{\overline{j}}^{0}\widetilde{W}_{%
\overline{\beta }\gamma }^{\overline{\alpha }}\right) \\ 
& \quad -\sum_{j,k,\alpha ,\beta } \left[ (\widetilde{A}_{\overline{\beta }%
}^{\alpha }f_{\overline{k}}^{\overline{\beta }}f_{j}^{0})_{k}f_{\overline{j}%
}^{\overline{\alpha }}+(\widetilde{A}_{\beta }^{\overline{\alpha }%
}f_{k}^{\beta }f_{\overline{j}}^{0})_{\overline{k}}f_{j}^{\alpha }+(%
\widetilde{A}_{\overline{\beta }}^{\alpha }f_{k}^{0}f_{\overline{k}}^{%
\overline{\beta }})_{j}f_{\overline{j}}^{\overline{\alpha }}+(\widetilde{A}%
_{\beta }^{\overline{\alpha }}f_{\overline{k}}^{0}f_{k}^{\beta })_{\overline{%
j}}f_{j}^{\alpha } \right]
\end{aligned}%
\label{equation3.20}
\end{equation}%
where $R_{k\overline{i}j\overline{k}}=R_{kj\overline{k}}^{i}$ and $%
\widetilde{R}_{\alpha \overline{\beta }\gamma \overline{\delta }}=\widetilde{%
R}_{\alpha \gamma \overline{\delta }}^{\beta }$.

An important special kind of $(H,\widetilde{H})$-holomorphic maps are \emph{%
CR} maps, which are defined as follows:

\begin{definition}[{\cite{[DT]}}] \label{definition3.2}
A smooth map $f:(M,H,J,\theta)\rightarrow (N,\widetilde{H},\widetilde{J},\widetilde{\theta })$ between
two pseudo-Hermitian manifolds is called a \emph{CR} map if it is $(H,%
\widetilde{H})$-holomorphic and horizontal. Here the horizontal condition
means that $df(H)\subset \widetilde{H}$.
\end{definition}

Now we want to deduce the Bochner formula of $e_{H,\widetilde{H}}(f)$ for a 
\emph{CR} map. Suppose $f:M\rightarrow N$ is a \emph{CR} map. Then, besides
\eqref{equation3.3}, it also satisfies 
\begin{equation}
f_{k}^{0}=f_{\overline{k}}^{0}=0.  \label{equation3.21}
\end{equation}%
According to \eqref{equation3.7} and \eqref{equation3.21}, we have%
\begin{equation}
f_{0l}^{0}=2\sqrt{-1}\sum_{\alpha }f_{l}^{\alpha }f_{0}^{\overline{\alpha }}
\label{equation3.22}
\end{equation}%
and 
\begin{equation}
f_{0}^{0}\delta _{j}^{i}=\sum_{\alpha }f_{j}^{\alpha }f_{\overline{i}}^{%
\overline{\alpha }}.  \label{equation3.23}
\end{equation}%
In particular, \eqref{equation3.23} gives%
\begin{equation}
f_{0}^{0}=\frac{1}{m}e_{H,\widetilde{H}}(f).  \label{equation3.24}
\end{equation}%
From \eqref{equation3.22} and the fourth equation in \eqref{equation3.11}, it follows that%
\begin{equation*}
f_{0l\overline{l}}^{0}=-4\sum_{\alpha }\mid f_{0}^{\alpha }\mid ^{2}+2\sqrt{%
-1}\sum_{\alpha }f_{l}^{\alpha }f_{0\overline{l}}^{\overline{\alpha }}
\end{equation*}%
which, combining with \eqref{equation3.24}, implies%
\begin{equation}
\frac{1}{m}\bigtriangleup _{b}e_{H,\widetilde{H}}=\bigtriangleup
_{b}f_{0}^{0}=-8m\sum_{\alpha }\mid f_{0}^{\alpha }\mid ^{2}+2\sqrt{-1}%
\sum_{l,\alpha }(f_{l}^{\alpha }f_{0\overline{l}}^{\overline{\alpha }}-f_{%
\overline{l}}^{\overline{\alpha }}f_{0l}^{\alpha })  \label{equation3.25}
\end{equation}%
From \eqref{equation3.11}, \eqref{equation3.20}, we get%
\begin{equation*}
\begin{aligned}
\frac{2m+2}{m}\bigtriangleup _{b}e_{H,\widetilde{H}} & =2\sum_{j,k,\alpha }\mid
f_{jk}^{\alpha }\mid ^{2}-8m(m+1)\sum_{\alpha }\mid f_{0}^{\alpha }\mid ^{2}
\\ 
& \quad +2\sum_{i,j,k,\alpha }f_{i}^{\alpha }f_{\overline{j}}^{\overline{\alpha }%
}R_{k\overline{i}j\overline{k}}-2\sum_{j,k,\alpha ,\beta ,\gamma ,\delta
}f_{k}^{\alpha }f_{\overline{j}}^{\overline{\beta }}f_{j}^{\gamma }f_{%
\overline{k}}^{\overline{\delta }}\widetilde{R}_{\alpha \overline{\beta }%
\gamma \overline{\delta }}%
\end{aligned}%
\end{equation*}%
Therefore we conclude that

\begin{lemma} \label{lemma3.1}
Let $f:(M^{2m+1},H,J,\theta )\rightarrow (N^{2n+1},%
\widetilde{H},\widetilde{J},\widetilde{\theta })$ be a \emph{CR} map between
two pseudo-Hermitian manifolds. Then%
\begin{equation}
\begin{aligned}
\frac{m+1}{m}\bigtriangleup _{b}e_{H,\widetilde{H}} & =\sum_{j,k,\alpha }\mid
f_{jk}^{\alpha }\mid ^{2}-4m(m+1)e_{L,\widetilde{H}} \\ 
& \quad +\sum_{i,j,k,\alpha }f_{i}^{\alpha }f_{\overline{j}}^{\overline{\alpha }}R_{k%
\overline{i}j\overline{k}}-\sum_{j,k,\alpha ,\beta ,\gamma ,\delta
}f_{k}^{\alpha }f_{\overline{j}}^{\overline{\beta }}f_{j}^{\gamma }f_{%
\overline{k}}^{\overline{\delta }}\widetilde{R}_{\alpha \overline{\beta }%
\gamma \overline{\delta }}%
\end{aligned}%
\label{equation3.26}
\end{equation}%
where $e_{L,\widetilde{H}}=\frac{1}{2}\mid \pi _{\widetilde{H}}\circ df\circ
i_{L}\mid ^{2}=\sum_{\alpha }\mid f_{0}^{\alpha }\mid ^{2}$, and $%
i_{L}:L\rightarrow TM$ is the inclusion morphism.
\end{lemma}

Note that for a map $f:(M,H,J,\theta )\rightarrow (N,\widetilde{H},%
\widetilde{J},\widetilde{\theta })$, the linear map $df_{H,\widetilde{H}%
}:H\rightarrow \widetilde{H}$ gives a dual map $f_{H,\widetilde{H}}^{\ast }:%
\widetilde{H}^{\ast }\rightarrow H^{\ast }$ at each point of $M$, which
induces a map from $\otimes ^{p}\widetilde{H}^{\ast }$ to $\otimes
^{p}H^{\ast }$ for any $p\geq 1$. For example, $f_{H,\widetilde{H}}^{\ast
}L_{\widetilde{\theta }}$ is a symmetric $2$-tensor field defined by%
\begin{equation}
\left( f_{H,\widetilde{H}}^{\ast }L_{\widetilde{\theta }}\right)
(X,Y)=L_{\theta }(df_{H,\widetilde{H}}(X),df_{H,\widetilde{H}}(Y)) 
\label{equation3.27}
\end{equation}%
for any $X,Y\in H$. If $f$ is a \emph{CR} map, then \eqref{equation3.21} implies 
\begin{equation}
f^{\ast }\widetilde{\theta }=f_{0}^{0}\theta  \label{equation3.28}
\end{equation}%
and thus%
\begin{equation}
f^{\ast }d\widetilde{\theta }=df_{0}^{0}\wedge \theta +f_{0}^{0}d\theta . 
\label{equation3.29}
\end{equation}%
When both pseudo-Hermitian manifolds $M$ and $N$ have the same dimension,
one may consider the ratio of the horizontal volume elements (resp. volume
elements) under the \emph{CR} map $f$, which is given by using \eqref{equation3.28} and
\eqref{equation3.29} that
\begin{equation*}
\frac{f_{H,\widetilde{H}}^{\ast }\Omega _{N}^{\widetilde{H}}}{\Omega _{M}^{H}
}=\frac{(f_{H,\widetilde{H}}^{\ast }\Omega _{N}^{\widetilde{H}})(\eta
_{1},...,\eta _{m},\eta _{\overline{1}},...,\eta _{\overline{m}})}{\Omega
_{M}^{H}(\eta _{1},...,\eta _{m},\eta _{\overline{1}},...,\eta _{\overline{m}
})}=(f_{0}^{0})^{m} , \qquad 
\left( \text{resp. }\frac{f^{\ast }\Omega _{N}}{\Omega _{M}}=(f_{0}^{0})^{m+1} \right) .
\end{equation*}
Hence, according to \eqref{equation3.24}, we find that the estimate for $e_{H,\widetilde{H}
}(f)$ will essentially give the estimates for the ratios of these volume
elements.

Another important special kind of $(H,\widetilde{H})$-holomorphic maps are
transverally holomorphic maps.

\begin{definition} \label{definition3.3}
A smooth map $f:(M,H,J,\theta )\rightarrow (N,%
\widetilde{H},\widetilde{J},\widetilde{\theta })$ between two
pseudo-Hermitian manifolds is called a \emph{transversally holomorphic map}
if it is $(H,\widetilde{H})$-holomorphic and foliated. Here the foliated
condition means that $df(L)\subset \widetilde{L}$.
\end{definition}

\begin{remark} \label{remark3.1}
Transversally holomorphic maps between K\"{a}hlerian
foliations (resp. transversally Hermitian manifolds) have been investigated
in \cite{[BD]} (resp. \cite{[KW]}). Transversally holomorphic maps between
pseudo-Hermitian manifolds in the sense of Definition \ref{definition3.3} were studied in
\cite{[Do]} under the name of foliated $(H,\widetilde{H})$-holomorphic maps too.
\end{remark}

In the remaining part of this section, we suppose $f$ is a transverally
holomorphic map between pseudo-Hermitian manifolds. Clearly the foliated
condition is equivalent to%
\begin{equation}
f_{0}^{\alpha }=f_{0}^{\overline{\alpha }}=0.  \label{equation3.30}
\end{equation}%
From \eqref{equation3.20} and \eqref{equation3.30}, we get immediately the following.

\begin{lemma} \label{lemma3.2}
Let $f:(M^{2m+1},H,J,\theta )\rightarrow (N^{2n+1},%
\widetilde{H},\widetilde{J},\widetilde{\theta })$ be a transverally
holomorphic map. If $N$ is Sasakian, then%
\begin{equation*}
\frac{1}{2} \bigtriangleup _{b}e_{H,\widetilde{H}}=\sum_{\alpha ,j,k}\mid f_{jk}^{\alpha
}\mid ^{2}+\sum_{i,j,k,\alpha }f_{i}^{\alpha }f_{\overline{j}}^{\overline{%
\alpha }}R_{k\overline{i}j\overline{k}}-\sum_{j,k,\alpha ,\beta ,\gamma
,\delta }f_{k}^{\alpha }f_{\overline{j}}^{\overline{\beta }}f_{j}^{\gamma
}f_{\overline{k}}^{\overline{\delta }}\widetilde{R}_{\alpha \overline{\beta }%
\gamma \overline{\delta }}.
\end{equation*}
\end{lemma}

Now we turn to deduce the Bochner formula for the ratio of the horizontal
volume elements under a transversally holomorphic map $f:M\rightarrow N$
with additional assumptions that $\dim M=\dim N=2m+1$ and $N$ is Sasakian.
By \eqref{equation2.9}, \eqref{equation2.10} and the first equation in \eqref{equation2.14}, the horizontal volume
elements of $M$ and $N$ are given, up to a constant, by%
\begin{equation*}
\Omega _{M}^{H}=\theta ^{1}\wedge \cdots \wedge \theta ^{m}\wedge \theta ^{%
\overline{1}}\wedge \cdots \wedge \theta ^{\overline{m}}
\end{equation*}%
and%
\begin{equation*}
\Omega _{N}^{H}=\widetilde{\theta }^{1}\wedge \cdots \wedge \widetilde{%
\theta }^{m}\wedge \widetilde{\theta }^{\overline{1}}\wedge \cdots \wedge 
\widetilde{\theta }^{\overline{m}}
\end{equation*}%
respectively. Using \eqref{equation3.5} and \eqref{equation3.30}, we obtain%
\begin{equation}
v=\frac{f^{\ast }\Omega _{N}^{H}}{\Omega _{M}^{H}}=\Phi \overline{\Phi } 
\label{equation3.31}
\end{equation}%
where $\Phi =\det \left( f_{i}^{\alpha }\right) $. By the assumptions for $f$
and $N$, we have from \eqref{equation3.11} and \eqref{equation3.30} that%
\begin{equation}
f_{j\overline{l}}^{\alpha }=0,\text{ \ }f_{j0}^{\alpha }=0.  \label{equation3.32}
\end{equation}%
Consequently \eqref{equation3.10} and \eqref{equation3.32} yield 
\begin{equation}
df_{j}^{\alpha }=\sum_{l}f_{jl}^{\alpha }\theta ^{l}+\sum_{k}f_{k}^{\alpha
}\theta _{j}^{k}-\sum_{\beta }f_{j}^{\beta }\widetilde{\theta }_{\beta
}^{\alpha }  \label{equation3.33}
\end{equation}%
Using \eqref{equation3.33}, we derive%
\begin{equation}
d\Phi =\sum_{l}\Phi _{l}\theta ^{l}+\sum_{\bar{l}}\Phi _{\overline{l}}\theta
^{\overline{l}}+\Phi _{0}\theta \\ 
=\sum_{l}\Psi _{l}\theta ^{l}+\Phi {\big(}\sum_{j}\theta
_{j}^{j}-\sum_{\alpha }\widetilde{\theta }_{\alpha }^{\alpha }{\big)}
\label{equation3.34}
\end{equation}%
where 
\begin{equation*}
\Psi _{l}=\Phi \sum_{\alpha ,j}b_{\alpha }^{j}f_{jl}^{\alpha }\text{, \ \ }%
(b_{\alpha }^{j})=(f_{j}^{\alpha })^{-1}.
\end{equation*}%
Combining \eqref{equation2.15} and \eqref{equation3.34}, we obtain%
\begin{equation}
dv=\overline{\Phi }d\Phi +\Phi d\overline{\Phi } \\ 
=\sum_{l}(\overline{\Phi }\Psi _{l}\theta ^{l}+\Phi \overline{\Psi _{l}}%
\theta ^{\overline{l}})
\label{equation3.35}
\end{equation}%
that is, $v_{l}=\overline{\Phi }\Psi _{l}$ and $v_{0}=0$. Taking the
exterior derivative of \eqref{equation3.34} and using the structure equations in both $M$
and $N$, we have%
\begin{equation}
\begin{aligned}
0=d^{2}\Phi &=\sum_{l}\left( d\Psi _{l}-\sum_{k}\Psi _{k}\theta _{l}^{k}-\Psi
_{l}(\sum_{i}\theta _{i}^{i}-\sum_{\alpha }\widetilde{\theta }_{\alpha
}^{\alpha })\right) \wedge \theta ^{l} \\ 
& \quad +\sum_{k}\Psi _{k}\theta \wedge \tau^{k} +\Phi \left(\sum_{i}d\theta _{i}^{i}-\sum_{\alpha }d\widetilde{\theta }_{\alpha
}^{\alpha } \right) \\ 
& =\sum_{l}\left( d\Psi _{l}-\sum_{k}\Psi _{k}\theta _{l}^{k}-\Psi
_{l}(\sum_{i}\theta _{i}^{i}-\sum_{\alpha }\widetilde{\theta }_{\alpha
}^{\alpha })\right) \wedge \theta ^{l}+\sum_{k,j}\Psi _{k}A_{\overline{j}%
}^{k}\theta \wedge \theta ^{\overline{j}} \\ 
& \quad +2\sqrt{-1}\Phi \sum_{i,j}(\theta ^{i}\wedge A_{j}^{\overline{i}}\theta
^{j}+2\sqrt{-1}\theta ^{\overline{i}}\wedge A_{\overline{j}}^{i}\theta ^{%
\overline{j}})+\sum_{i,k,l}R_{ik\overline{l}}^{i}\theta ^{k}\wedge \theta ^{%
\overline{l}} \\ 
 & \quad +\sum_{i,k}(W_{ik}^{i}\theta ^{k}-W_{i\overline{k}}^{i}\theta ^{\overline{k}%
})\wedge \theta -\Phi \sum_{i,j,\alpha ,\gamma ,\delta }f_{i}^{\gamma }f_{%
\overline{j}}^{\overline{\delta }}\widetilde{R}_{\alpha \gamma \overline{%
\delta }}^{\alpha }\theta ^{i}\wedge \theta ^{\overline{j}}%
\end{aligned}
\label{equation3.36}
\end{equation}%
We put 
\begin{equation}
d\Psi _{l}-\sum_{k}\Psi _{k}\theta _{l}^{k}-\Psi _{l}(\sum_{i}\theta
_{i}^{i}-\sum_{\alpha }\widetilde{\theta }_{\alpha }^{\alpha })=\psi
_{l0}\theta +\sum_{k}(\psi _{lk}\theta ^{k}+\psi _{l\overline{k}}\theta ^{%
\overline{k}}).  \label{equation3.37}
\end{equation}%
Substituting \eqref{equation3.37} into \eqref{equation3.36} gives%
\begin{equation}
\psi _{k\overline{l}}=\Phi R_{k\overline{l}}-\Phi \sum_{\gamma ,\delta }%
\widetilde{R}_{\gamma \overline{\delta }}f_{k}^{\gamma }f_{\overline{l}}^{%
\overline{\delta }}.  \label{equation3.38}
\end{equation}%
According to \eqref{equation3.34}, \eqref{equation3.35} and \eqref{equation3.37}, we deduce that%
\begin{equation}
dv_{k}-\sum_{j}v_{j}\theta _{k}^{j}=\sum_{l}(\Psi _{k}\overline{\Psi _{l}}+%
\overline{\Phi }\psi _{k\overline{l}})\theta ^{\overline{l}}+\sum_{l}%
\overline{\Phi }\psi _{kl}\theta ^{l}+\overline{\Phi }\psi _{k0}\theta . 
\label{equation3.39}
\end{equation}%
It follows from \eqref{equation3.38} and \eqref{equation3.39} that%
\begin{equation}
v_{k\overline{l}}=\Psi _{k}\overline{\Psi _{l}}+\overline{\Phi }\psi _{k%
\overline{l}}=\Psi _{k}\overline{\Psi _{l}}+v(R_{k\overline{l}}-\widetilde{R}%
_{\gamma \overline{\delta }}f_{k}^{\gamma }f_{\overline{l}}^{\overline{%
\delta }}).  \label{equation3.40}
\end{equation}%
Therefore we get immediately from \eqref{equation2.20}, \eqref{equation2.21} and \eqref{equation3.40} the following
lemma.

\begin{lemma} \label{lemma3.3}
Let $f:M\rightarrow N$ be a transversally holomorphic
map between two pseudo-Hermitian manifolds of the same dimension $2m+1$. If $%
N$ is Sasakian, one has%
\begin{equation}
\frac{1}{2}\triangle _{b}v=\sum_{k}\Psi _{k}\overline{\Psi _{k}}%
+v \left(R-\sum_{k,\gamma ,\delta }\widetilde{R}_{\gamma \overline{\delta }%
}f_{k}^{\gamma }f_{\overline{k}}^{\overline{\delta }} \right)  \label{equation3.41}
\end{equation}%
where $\Psi _{k}\ $is given by \eqref{equation3.34}. In particular, if $v>0$, then%
\begin{equation}
\frac{1}{2}\triangle _{b}\log v=R-\sum_{k,\gamma ,\delta }\widetilde{R}%
_{\gamma \overline{\delta }}f_{k}^{\gamma }f_{\overline{k}}^{\overline{%
\delta }}.  \label{equation3.42}
\end{equation}
\end{lemma}

\section{Schwarz lemmas for \emph{CR} maps} \label{section4}

First, we recall the following sub-Laplacian comparison result for
pseudo-Hermitian manifolds.

\begin{lemma}[\cite{[CDRZ]}] \label{lemma4.1}
Let $(M^{2m+1},H,J,\theta )$ be a complete
pseudo-Hermitian manifold with pseudo-Hermitian Ricci curvature bounded from
below by $-mk_{1}\leq 0$ and $\Vert A\Vert _{C^{1}}\ $bounded from above by $%
k$ ($k\geq 0$). Let $r$ be the Riemannian distance of $g_{\theta }$ relative
to a fixed point $p$. Then if $x\ $is not on the cut locus of $p$, we have%
\begin{equation*}
\bigtriangleup _{b}r(x)\leq C\left( \frac{1}{r}+\sqrt{1+k_{1}+k+k^{2}}\right)
\end{equation*}%
where $C$ is a constant depending only on $m$.
\end{lemma}

\begin{remark} \label{remark4.1}
For a pseudo-Hermitian manifold, there are two notions
of completeness, which are defined respectively by the Carnot-Carath\'{e}%
odory distance $r_{CC}$ of $L_{\theta }$ and the Riemannian distance $r$ of $%
g_{\theta }$. Actually they are equivalent, since these two distance
functions are locally controlled by each other (cf. for example, \cite{[NSW]}).
\end{remark}

\begin{example} \label{example4.1}
Let $(M,H,J,\theta )$ be a complete Sasakian manifold
with pseudo-Hermitian Ricci curvature bounded from below (e.g. the
Heisenberg group). Set $\widehat{\theta }=e^{2\varphi }\theta $ with $%
\varphi \in C^{\infty }(M)$. By Lemma 2.7 in \cite{[DT]}, we know that the
pseudo-Hermitian torsion of $(M,H,J,\widehat{\theta })$ is given by%
\begin{equation}
\widehat{A}_{ij}=\sqrt{-1}(\varphi _{ij}-\varphi _{i}\varphi _{j}). 
\label{equation4.1}
\end{equation}%
In terms of \eqref{equation2.16} and \eqref{equation4.1}, we find that if $\Vert \varphi \Vert _{C^{3}}$
is bounded from above, then $(M,H,J,\widehat{\theta })$ is a complete
pseudo-Hermitian manifold with pseudo-Hermitian Ricci curvature bounded from
below and $\Vert \widehat{A}\Vert _{C^{1}}$ bounded from above. Therefore we
may construct many examples of pseudo-Hermitian manifolds satisfying the
conditions in Lemma \ref{lemma4.1}.
\end{example}

Now we recall the notion of dilatation for a map between two Riemannian
manifolds. Suppose $A:V_{1}\rightarrow V_{2}$ is a linear map between two
Euclidean linear spaces. Let 
\begin{equation*}
\lambda _{1}\geq \lambda _{2}\geq ...\geq \lambda _{l}\geq 0
\end{equation*}%
be the eigenvalues of the positive semidefinite symmetric matrix $A^{t}A$,
where $l=\dim V_{1}$ and $A^{t}$ denotes the transpose of $A$. If there is a
positive number $\Lambda $ such that 
\begin{equation}
\lambda _{1}\leq \Lambda ^{2}\left( \sum_{k=2}^{l}\lambda _{k}\right) , 
\label{equation4.2}
\end{equation}%
one says that the generalized dilatation of $A$ is less that or equal to $%
\Lambda $. Note that \eqref{equation4.2} is equivalent to%
\begin{equation}
\lambda _{1}\leq \frac{\Lambda ^{2}}{1+\Lambda ^{2}}\left(
\sum_{k=1}^{l}\lambda _{k}\right) .  \label{equation4.3}
\end{equation}%
If $A\neq 0$ (i.e., $\lambda _{1}>0$), then the condition \eqref{equation4.2} implies that%
\begin{equation}
\Lambda \geq \frac{1}{\sqrt{l-1}}.  \label{equation4.4}
\end{equation}%
According to \cite{[Sh]}, we introduce the following

\begin{definition} \label{definition4.1}
Let $f:(M,g)\rightarrow (N,\widetilde{g})$ be a $%
C^{\infty }$ map between two Riemannian manifolds. For a given positive
number $\Lambda $, if the genearalized dilatation of $df_{x}:(T_{x}M,g_{x})%
\rightarrow (T_{f(x)},\widetilde{g}_{f(x)})$ at each point $x\in M$ is$\ $%
less than or equal to $\Lambda $, then we say that $f$ has bounded
generalized dilatation of order $\Lambda $.
\end{definition}

\begin{remark} \label{remark4.2}
Note that the notion of $\Lambda $-quasiconformal maps
introduced by S.I. Goldberg and T. Ishihara in \cite{[GI]} is defined for maps with
constant maximal rank, but maps of bounded (generalized) dilatation (cf.
\cite{[GIP]}, \cite{[Sh]}) are not required to have this rank condition. According to
\cite{[GIP]}, we know that a $\Lambda $-quasiconformal is a map of bounded
dilatation of order $\Lambda $, which means that there exists a positive
constant $\Lambda ^{2}$ such that $\lambda _{1}(df_{x})\leq \Lambda
^{2}\lambda _{s}(df_{x})$ for some $s$ ($2\leq s\leq l=\dim M$) at each
point $x\in M$. Clearly a map of bounded dilation of order $\Lambda $ must
be a map of bounded generalized dilatation of order $\Lambda $ in the sense
of \cite{[Sh]}. So we can summarize the relations between these notions as follows: 
$\Lambda $-quasiconformal $\Rightarrow $ bounded dilatation of order $%
\Lambda $ $\Rightarrow $ bounded generalized dilatation of order $\Lambda $.
By definition, if $\Lambda _{1}\leq \Lambda _{2}$ and $f$ has bounded
generalized dilatation of order $\Lambda _{1}$, then $f$ is clearly a map
with bounded generalized dilatation of order $\Lambda _{2}$.
\end{remark}

For a map $f:(M,H,J,\theta )\rightarrow (N,\widetilde{H},\widetilde{J},%
\widetilde{\theta })$ between two pseudo-Hermitian manifolds, the notion of
generalized dilatation will be defined with respect to $g_{\theta }$ and $g_{%
\widetilde{\theta }}$. Now we have

\begin{theorem} \label{theorem4.2}
Let $(M^{2m+1},H,J,\theta )$ be a complete
pseudo-Hermitian manifold with pseudo-Hermitian Ricci curvature bounded from
below by $-mk_{1}$ ($k_{1}\geq 0$) and $\Vert A\Vert _{C^{1}}\ $bounded from
above. Let $(N^{2n+1},\widetilde{H},\widetilde{J},\widetilde{\theta })$ be
another pseudo-Hermitian manifold with pseudo-Hermitian bisectional
curvature bounded from above by $-k_{2}$ ($k_{2}>0$). Then for any \emph{CR}
map $f:M\rightarrow N$ with bounded generalized dilatation of order $\Lambda 
$, we have%
\begin{equation}
f_{H,\widetilde{H}}^{\ast }L_{\widetilde{\theta }}\leq \frac{4(m+1)\Lambda
^{2}+k_{1}}{k_{2}}L_{\theta }.  \label{equation4.5}
\end{equation}%
In particular, if $k_{2}\geq 4(m+1)\Lambda ^{2}+k_{1}$, then $f$ is
horizontally distance decreasing.
\end{theorem}

\begin{proof}
By \eqref{equation3.29}, we have%
\begin{equation*}
f_{H,\widetilde{H}}^{\ast }L_{\widetilde{\theta }}=f_{0}^{0}L_{\theta }.
\end{equation*}%
Accordong to \eqref{equation3.24}, it suffices to estimate the function $u=e_{H,\widetilde{%
H}}(f)$. Let $r(x)$ be the Riemannian distance function of $g_{\theta }$
relative to a fixed point $p\in M$ and let $B_{p}(a)=\{x\in M | r(x)<a\}$. 
Consider the following function%
\begin{equation*}
\phi (x)={\big(}a^{2}-r^{2}(x){\big)}^{2}u(x)
\end{equation*}%
on $B_{p}(a)$. Without loss of generality, we assume that $u\not\equiv 0$ on 
$B_{p}(a)$.
By computing the gradient and sub-Laplacian of $\phi $ at the
maximum point $x_{0}$, we have
\begin{gather}
\frac{\nabla ^{H}u}{u}=-2\frac{\nabla^H (a^{2}-r^{2})}{a^{2}-r^{2}},  \label{equation4.6} \\
\frac{\bigtriangleup _{b}u}{u}-\frac{|\nabla ^{H}u|^{2}}{u^{2}}+2\frac{%
\bigtriangleup _{b}(a^{2}-r^{2})}{a^{2}-r^{2}}-2\frac{|\nabla
^{H}(a^{2}-r^{2})|^{2}}{(a^{2}-r^{2})^{2}}\leq 0.  \label{equation4.7}
\end{gather}%
Here $r$ is assumed to be smooth near $x_0$, otherwise it can be modified as usual (cf. \cite{[Che]}).
Using \eqref{equation4.6} and \eqref{equation4.7}, we find that
\begin{equation}
\frac{\bigtriangleup _{b}u}{u}+2\frac{\bigtriangleup _{b}(a^{2}-r^{2})}{%
a^{2}-r^{2}}-6\frac{|\nabla ^{H}(a^{2}-r^{2})|^{2}}{(a^{2}-r^{2})^{2}}\leq 0.
\label{equation4.8}
\end{equation}%
Clearly%
\begin{equation}
\mid \nabla ^{H}(a^{2}-r^{2})\mid ^{2}=\mid 2r\nabla ^{H}r\mid ^{2}\leq
4a^{2}.  \label{equation4.9}
\end{equation}%
It follows from Lemma \ref{lemma4.1} that%
\begin{equation}
\bigtriangleup _{b}r^{2}=2r\bigtriangleup _{b}r+2\mid \nabla ^{H}r\mid
^{2}\leq C(1+a)  \label{equation4.10}
\end{equation}%
where $C$ is a positive constant independent of $a$. Consequently 
\begin{align}
\frac{\bigtriangleup _{b}u}{u} & \leq 2\frac{\bigtriangleup _{b}r^{2}}{%
a^{2}-r^{2}}+6\frac{|\nabla ^{H}(a^{2}-r^{2})|^{2}}{(a^{2}-r^{2})^{2}} \nonumber \\
& =2\frac{\bigtriangleup _{b}r^{2}}{a^{2}-r^{2}}+24\frac{r^{2}}{%
(a^{2}-r^{2})^{2}} 
\leq \frac{C(1+a)}{a^{2}-r^{2}(x_{0})}+\frac{24a^{2}}{\left(
a^{2}-r^{2}(x_{0})\right) ^{2}}.%
\label{equation4.11}
\end{align}%
Using \eqref{equation3.26} and the assumptions about the curvatures of both $M$ and $N$,
we have%
\begin{equation}
\frac{2m+2}{m}\bigtriangleup _{b}u\geq -8m(m+1)e_{L,\widetilde{H}%
}(f)-2mk_{1}u+2k_{2}u^{2}  \label{equation4.12}
\end{equation}%
Since $f$ has bounded generalized dilation of order $\Lambda $, one has%
\begin{equation*}
|df(\xi )|^{2}\leq \frac{\Lambda ^{2}}{1+\Lambda ^{2}}\left( |df(\xi
)|^{2}+\mid df(\eta _{i})\mid ^{2}+\mid df(\eta _{\overline{i}})\mid
^{2}\right) 
=\frac{\Lambda ^{2}}{1+\Lambda ^{2}}\mid df(\xi )\mid ^{2}+\frac{2\Lambda
^{2}u}{1+\Lambda ^{2}}%
\end{equation*}%
which implies that 
\begin{equation}
2e_{L,\widetilde{H}}(f)=|\pi _{\widetilde{H}}df(\xi )|^{2} 
\leq |df(\xi )|^{2} 
\leq 2\Lambda ^{2}u.%
\label{equation4.13}
\end{equation}%
Then \eqref{equation4.13} gives 
\begin{equation}
\frac{\bigtriangleup _{b}u}{u}\geq -(4m^{2}\Lambda ^{2}+\frac{m^{2}k_{1}}{m+1%
})+\frac{mk_{2}}{m+1}u  \label{equation4.14}
\end{equation}%
From \eqref{equation4.11} and \eqref{equation4.14}, we have%
\begin{equation*}
u(x_{0})\leq \frac{4m(m+1)\Lambda ^{2}+mk_{1}}{k_{2}}+\frac{m+1}{mk_{2}}%
\left( \frac{C(1+a)}{a^{2}-r^{2}(x_{0})}+\frac{24a^{2}}{\left(
a^{2}-r^{2}(x_{0})\right) ^{2}}\right) .
\end{equation*}%
Therefore%
\begin{align*}
(a^{2}-r^{2}(x))^{2}u(x) &\leq (a^{2}-r^{2}(x_{0}))^{2}u(x_{0}) 
\\
&\leq \frac{4m(m+1)\Lambda ^{2}+mk_{1}}{k_{2}}a^{4} +\frac{m+1}{mk_{2}} \left[ C(1+a)a^{2}+24a^{2} \right]
\end{align*}%
for any $x\in B_{p}(a)$. It follows that 
\begin{equation}
u(x)\leq \frac{4m(m+1)\Lambda ^{2}+mk_{1}}{k_{2}} \ \frac{a^{4}}{(a^{2}-r^{2}(x))^{2}}+\frac{m+1}{mk_{2}} \ \frac{C(1+a)a^{2}+24a^{2}}{(a^{2}-r^{2}(x))^{2}}  \label{equation4.16}
\end{equation}%
By taking $a\rightarrow \infty $ in \eqref{equation4.16}, we get%
\begin{equation*}
u(x)\leq \frac{4m(m+1)\Lambda ^{2}+mk_{1}}{k_{2}}.
\end{equation*}%
It then follows from \eqref{equation3.24} that $f_{0}^{0}\leq \frac{4(m+1)\Lambda
^{2}+k_{1}}{k_{2}}$. Therefore we complete the proof of this theorem.
\end{proof}

In case $m=1$, one can weaken the hypothesis on $N$.

\begin{corollary} \label{corollary4.3}
Let $(M^{3},H,J,\theta )$ be a complete
pseudo-Hermitian manifold with pseudo-Hermitian curvature bounded from below
by $-k_{1}$ ($k_{1}\geq 0$) and $\Vert A\Vert _{C^{1}}\ $bounded from above.
Let $(N^{2n+1},\widetilde{H},\widetilde{J},\widetilde{\theta })$ be another
pseudo-Hermitian manifold with pseudo-Hermitian sectional curvature bounded
from above by $-k_{2}$ ($k_{2}>0$). Then for any \emph{CR} map $%
f:M\rightarrow N$ with bounded generalized dilatation of order $\Lambda $, we have
\begin{align*}
f_{H,\widetilde{H}}^{\ast }L_{\widetilde{\theta }}\leq \frac{8 \Lambda^{2}+k_{1}}{k_{2}}L_{\theta }.
\end{align*}
\end{corollary}

In general we need the following lemma established by Royden to weaken the
curvature condition on $N$.

\begin{lemma}[{\cite{[Ro]}}] \label{lemma4.4}
Let $\zeta _{1},...,\zeta _{\nu }$ be orthogonal
vectors in a complex linear space $V$ endowed with a Hermitian inner
product. If $S(\zeta ,\overline{\eta },\sigma ,\overline{\omega })$ is a
symmetric bihermitian form on $V$, that is, $S(\zeta ,\overline{\eta }%
,\sigma ,\overline{\omega })=S(\sigma ,\overline{\eta },\zeta ,\overline{%
\omega })$ and $S(\eta ,\overline{\zeta },\omega ,\overline{\sigma })=%
\overline{S}(\zeta ,\overline{\eta },\sigma ,\overline{\omega })$, such that
for all $\zeta $%
\begin{equation*}
S(\zeta ,\overline{\zeta },\zeta ,\overline{\zeta })\leq K\parallel \zeta
\parallel ^{4},
\end{equation*}%
then 
\begin{equation*}
\sum_{\alpha ,\beta }S(\zeta _{\alpha },\overline{\zeta _{\alpha }},\zeta
_{\beta },\overline{\zeta _{\beta }})\leq \frac{1}{2}K \bigg[ \left(\sum_\alpha
\parallel \zeta _{\alpha }\parallel ^{2} \right)^{2}+\sum_\alpha \parallel \zeta _{\alpha
}\parallel ^{4} \bigg]
\end{equation*}%
If $K\leq 0$, then 
\begin{equation*}
\sum_{\alpha ,\beta }S(\zeta _{\alpha },\overline{\zeta _{\alpha }},\zeta
_{\beta },\overline{\zeta _{\beta }})\leq \frac{\nu +1}{2\nu }K \left(\sum_\alpha
\parallel \zeta _{\alpha }\parallel ^{2} \right)^{2}.
\end{equation*}
\end{lemma}

Since the curvature tensor of the Tanaka-Webster connection on a
pseudo-Hermitian manifold is a bihermitian form on the horizontal
distribution, one may establish a connection between the pseudo-Hermitian
sectional curvature and pseudo-Hermitian bisectional curvature by means of
Lemma \ref{lemma4.4}. Consequently we can replace the upper bound of the
pseudo-Hermitian bisectional curvature in Theorem \ref{theorem4.2} by that of the
pseudo-Hermitian sectional curvatue for the target pseudo-Hermitian manifold.

\begin{theorem} \label{theorem4.5}
Let $(M^{2m+1},H,J,\theta )$ be a complete
pseudo-Hermitian manifold wih pseudo-Hermitian Ricci curvature bounded from
below by $-mk_{1}$ and $\parallel A\parallel _{C^{1}}$ bounded from above.
Let $(N^{2n+1},\widetilde{H},\widetilde{J},\widetilde{\theta })$ be another
pseudo-Hermitian manifold with pseudo-Hermitian sectional curature bounded
from above by $-k_{2}<0$. Then for any \emph{CR} map $f:M\rightarrow N$ with
bounded generalized dilatation of order $\Lambda $, we have%
\begin{equation*}
f_{H,\widetilde{H}}^{\ast }L_{\widetilde{\theta }}\leq \frac{2\nu }{\nu +1}%
\frac{4(m+1)\Lambda ^{2}+k_{1}}{k_{2}}L_{\theta }
\end{equation*}%
where $\nu $ is the maximal rank of $df_{H,\widetilde{H}}$.
\end{theorem}

Now we assume that $\dim M=\dim N=2m+1$. Clearly for a horizontal map $%
f:M\rightarrow N$, one has%
\begin{equation*}
f^{\ast }\Omega _{N}=(f_{0}^{0})^{m+1}\Omega _{M}
\end{equation*}%
where $\Omega _{M}=\theta \wedge (d\theta )^{m}$ and $\Omega _{N}=\widetilde{%
\theta }\wedge (d\widetilde{\theta })^{m}$. It is also to see that%
\begin{equation*}
f_{H,\widetilde{H}}^{\ast }(d\widetilde{\theta })^{m}=(f_{0}^{0})^{m}(d%
\theta )^{m}.
\end{equation*}

\begin{corollary} \label{corollary4.6}
Let $M$ and $N$ be as in Theorem \ref{theorem4.2}. If $\dim
M=\dim N$, then for any \emph{CR} map $f:M\rightarrow N$ with bounded
generalized dilatation of order $\Lambda $, we have%
\begin{equation*}
f^{\ast }\Omega _{N}\leq \left( \frac{4(m+1)\Lambda ^{2}+k_{1}}{k_{2}}%
\right) ^{m+1}\Omega _{M}
\end{equation*}%
and%
\begin{equation*}
f_{H,\widetilde{H}}^{\ast }\Omega _{N}^{H}\leq \left( \frac{4(m+1)\Lambda
^{2}+k_{1}}{k_{2}}\right) ^{m}\Omega _{M}^{H}.
\end{equation*}
\end{corollary}

\section{Schwarz lemmas for transversally holomorphic maps} \label{section5}

In this section, we will establish Schwarz lemmas for transversally
holomorphic maps between pseudo-Hermitiant manifolds. According to \cite{[Do]}, we
introduce

\begin{definition} \label{definition5.1}
A map $f:(M,H,J,\theta )\rightarrow (N,\widetilde{H}%
,\widetilde{J},\widetilde{\theta })$ is called horizontally constant if it
maps the domain manifold into a single leaf of the pseudo-Hermitian
foliation on $N$.
\end{definition}

\begin{lemma}[{\cite{[Do]}}] \label{lemma5.1}
Let $f:(M,H,J,\theta )\rightarrow (N,\widetilde{H},%
\widetilde{J},\widetilde{\theta })$ be a map between two pseudo-Hermitian
manifolds. Then $f$ is horizontally constant if and only if $df_{H,%
\widetilde{H}}\equiv 0$.
\end{lemma}

\begin{proof}
If $f$ is horizontally constant, then $df(X)$ is
tangent to the fiber of $N$ for any $X\in TM$. Clearly we have $df_{H,%
\widetilde{H}}=0$.
Conversely, we assume that $df_{H,\widetilde{H}}=0$, which is equivalent to $%
f_{j}^{\alpha }=f_{\overline{j}}^{\overline{\alpha }}=f_{\overline{j}%
}^{\alpha }=f_{j}^{\overline{\alpha }}=0$. Then the fourth equation in
\eqref{equation3.11} yields $f_{0}^{\alpha }=f_{0}^{\overline{\alpha }}=0$. Hence $\pi _{%
\widetilde{H}}df(X)=0$ for any $X\in TM$. Suppose that $f(p)=q$ and $%
\widetilde{C}_{q}$ is the integral curve of $\widetilde{\xi }$ passing
through $q$. For any $p^{\prime }\in M$, let $c(t)$ be a smooth curve
joining $p$ and $p^{\prime }$. Obviuosly $f(c(t))$ is a smooth curve passing
through $p$ and $df(c^{\prime }(t))=\lambda (t)\widetilde{\xi }$ for some
function $\lambda (t)$, which means that $f(c(t))$ is the reparametrization
of the integral curve of $\widetilde{\xi }$. Therefore $f(c(t))\subset 
\widetilde{C}_{q}$. In particular, $f(p^{\prime })\in \widetilde{C}_{q}$.
Since $p^{\prime }$ is arbitrary, we conclude that $f(M)\subset \widetilde{C}%
_{q}$.
\end{proof}

First we give the following Schwarz type lemma for transversally holomorphic
maps, which may be regarded as a generalization of Yau's result.

\begin{theorem} \label{theorem5.2}
Let $(M^{2m+1},H,J,\theta )$ be a complete
pseudo-Hermitian manifold with pseudo-Hermitian Ricci curvature bounded from
below by $-mk_{1}\leq 0$ ($k_{1}\geq 0$) and $\parallel A\parallel _{C^{1}}$
bounded from above. Let $(N^{2n+1},\widetilde{H},\widetilde{J},\widetilde{%
\theta })$ be a Sasakian manifold with pseudo-Hermitian bisectional
curvature bounded from above by $-k_{2}<0$. Then any transversally
holomorphic map $f:M\rightarrow N$ satisfies%
\begin{equation*}
f^{\ast }G_{\widetilde{\theta }}\leq \frac{mk_{1}}{k_{2}}G_{\theta }.
\end{equation*}%
In particular, if $k_{1}=0$, then every transversally holomorphic map is
horizontally constant.
\end{theorem}

\begin{proof}
Set $u=e_{H,\widetilde{H}}(f)$. It is clear that%
\begin{equation}
f^{\ast }G_{\widetilde{\theta }}\leq uG_{\theta }.  \label{equation5.1}
\end{equation}%
From Lemma \ref{lemma3.2} and the curvature assumptions for both $M$ and $N$, we have%
\begin{equation}
\bigtriangleup _{b}u\geq -2mk_{1}u+2k_{2}u^{2}.  \label{equation5.2}
\end{equation}%
We consider the following function%
\begin{equation*}
\phi (x)={\big(}a^{2}-r^{2}(x){\big)}^{2}u(x)
\end{equation*}%
on $B_{p}(a)$, and assume that $u\not\equiv 0$ on $B_{p}(a)$.\ Similar to
the proof of Theorem \ref{theorem4.2}, we may use \eqref{equation5.2} to deduce at the maximum point $%
x_{0}$ of $\phi $ that 
\begin{equation*}
u(x_{0})\leq \frac{mk_{1}}{k_{2}}+\frac{12a^{2}}{%
k_{2}(a^{2}-r^{2}(x_{0}))^{2}}+\frac{C(1+a)}{k_{2}(a^{2}-r^{2}(x_{0}))}.
\end{equation*}%
Consequently%
\begin{align*}
(a^{2}-r^{2}(x))^{2}u(x) & \leq (a^{2}-r^{2}(x_{0}))^{2}u(x_{0}) \\ 
& \leq \frac{mk_{1}}{k_{2}}(a^{2}-r^{2}(x_{0}))^{2}+\frac{12a^{2}}{k_{2}}+%
\frac{C(1+a)}{k_{2}}(a^{2}-r^{2}(x_{0})) \\ 
& \leq \frac{mk_{1}}{k_{2}}a^{4}+\frac{12a^{2}}{k_{2}}+\frac{C(1+a)a^{2}}{k_{2}%
}
\end{align*}%
for any $x\in B_{p}(a)$. Therefore we get%
\begin{equation}
u(x)\leq \frac{mk_{1}a^{4}}{k_{2}(a^{2}-r^{2}(x))^{2}}+\frac{24a^{2}}{%
k_{2}(a^{2}-r^{2}(x))^{2}}+\frac{2C(1+a)a^{2}}{k_{2}(a^{2}-r^{2}(x))^{2}}. 
\label{equation5.3}
\end{equation}%
Letting $a\rightarrow \infty $ in \eqref{equation5.3}, we have%
\begin{equation*}
\sup_{M}u(x)\leq \frac{mk_{1}}{k_{2}}\text{.}
\end{equation*}%
If $k_{1}=0$, then $u\equiv 0$, which implies by Lemma \ref{lemma5.1} that $f$ is
horizontally constant.
\end{proof}

Similar to Theorem \ref{theorem5.2}, we have

\begin{theorem} \label{theorem5.3}
Let $(M^{2m+1},H,J,\theta )$ be a complete
pseudo-Hermitian manifold with pseudo-Hermitian Ricci curvature bounded from
below by $-mk_{1}\leq 0$ ($k_{1}\geq 0$) and $\Vert A\Vert _{C^{1}}$ bounded
from above. Let $(N^{2n+1},\widetilde{H},\widetilde{J},\widetilde{\theta })$
be a Sasakian manifold with pseudo-Hermitian sectional curvature bounded
from above by $-k_{2}<0$. Then for any transversally holomorphic map $%
f:M\rightarrow N$, one has%
\begin{equation*}
f^{\ast }G_{\theta }\leq \frac{2\nu }{\nu +1}\frac{mk_{1}}{k_{2}}G_{\theta }
\end{equation*}%
where $\nu $ is the maximal rank of $df_{H,\widetilde{H}}$. In particular,
if $k_{1}=0$, then every transversally holomorphic map is horizontally
constant.
\end{theorem}

In Theorem \ref{theorem5.2}, the curvature condition for $M$ is quite weak so that the
upper bound $mk_{1}/k_{2}$ for $f^{\ast }G_{\theta }$ is not precise enough.
In following, we will use the technique in \cite{[CCL]} to improve this upper
bound. First we recall the Hessian comparison theorem in Riemannian geometry.

\begin{lemma}[{\cite{[GW]}}] \label{lemma5.4}
Let $(M^{n},g)$ be a Riemannian manifold with
sectional curvature bounded from below by a constant $-K_{0}^{2}$. Let $r(x)$
be the Riemannian distance of $M$ relative to $p$. Then there exists a
constant $C=C(n,K_{0})$ such that%
\begin{equation*}
\left( Hess_{g}(r^{2})_{ij}\right) \leq CK_{0}^{2}(1+r)(\delta _{ij})
\end{equation*}%
holds ourside the cut locus of $p$, where the l.h.s. denotes the Riemannian
Hessian matrix of $r^{2}$ with respect to an orthonormal basis of $(M,g)$.
\end{lemma}

Note that the pseudo-Hermitian sectional curvatrue introduced in \cite{[We]} for a
pseudo-Hermitian manifold $(M,H,J,\theta )$ is a (pseudohermitian) analog of
the holomorphic sectional curvature of a Hermitian manifold rather than an
analog of the sectional curvature of a Riemannian manifold. Let $R(\cdot
,\cdot ,\cdot ,\cdot )$ be the curvature tensor of the Tanaka-Webster
connection $\nabla $ on $M$. E. Barletta investigated in \cite{[Ba]} the following
sectional curvature of Tanaka-Webster connection: 
\begin{equation}
K_{\sec }(\sigma )=R(u,v,u,v)  \label{equation5.4}
\end{equation}%
for any $2$-plane $\sigma \subset T_{x}M$, where $\{u,v\}$ is a $g_{\theta }$%
-orthonomal basis of $\sigma $. In particular, if $\sigma \subset H$, this
sectional curvature will be called the horizontal sectional curvature and
denoted by $K_{\sec }^{H}(\sigma )$. Now, by using the relationship between
the Levi-Civita connection and Tanaka-Webster connection of a
pseudo-Hermitian manifold, we can transform Lemma \ref{lemma5.4} into the following
form.

\begin{lemma} \label{lemma5.5}
Let $(M^{2m+1},H,J,\theta )$ be a pseudo-Hermitian
manifold with $K_{\sec }^{H}$ bounded below by $-\delta _{1}\leq 0$ and its
torsion satisfying $\Vert A\Vert _{C^{1}}\leq \delta _{2}$. Let $r(x)$ be
the Riemannian distance of $(M,g_{\theta })$ relative to $p$. Then there
exists a positive constant $D=D(m,\delta _{1},\delta _{2})$ such that%
\begin{equation*}
(r^{2})_{k\overline{l}}\zeta _{k}\zeta _{\overline{l}}+(r^{2})_{\overline{k}%
l}\zeta _{\overline{k}}\zeta _{l}\leq D(1+r)
\end{equation*}%
holds ourside the cut locus of $p$, where $\zeta =(\zeta _{1},...,\zeta
_{m})\in C^{m}$ is any unitary vector and $\left( r^{2}\right) _{k\overline{l%
}}=(\nabla dr^{2})(\eta _{k},\eta _{\overline{l}})$ denotes the complex
Hessian matrix of $r^{2}$ defined with respect to the Tanaka-Webster
connection.
\end{lemma}

\begin{proof}
Let $\nabla ^{\theta }$ be the Riemannian connection of the
Webster metric $g_{\theta }$, and $R^{\theta }(X,Y)Z$ be its curvature
tensor. The relationship between two connections $\nabla ^{\theta }$ and $%
\nabla $ is given by (cf. Lemma 1.3 in \cite{[DT]}):%
\begin{equation}
\nabla ^{\theta }=\nabla -(d\theta +A)\otimes \xi +\tau \otimes \theta
+2\theta \odot J,  \label{equation5.5}
\end{equation}%
where $2(\theta \odot J)(X,Y)=\theta (X)JY+\theta (Y)JX$ for any $X,Y\in TM$%
. \ This leads to the curvature relation between these two connections (cf.
Theorem 1.6 on page 49 of \cite{[DT]}):%
\begin{align}
\begin{aligned}
R^{\theta }(X,Y)Z& =R(X,Y)Z+(LX\wedge LY)Z+2d\theta (X,Y)JZ \\ 
& \quad -\langle S(X,Y),Z\rangle \xi +\theta (Z)S(X,Y) \\ 
& \quad -2\langle \theta \wedge \mathcal{O}\rangle (X,Y),Z\rangle \xi +2\theta
(Z)(\theta \wedge \mathcal{O})(X,Y)
\end{aligned}
\label{equation5.6}
\end{align}%
for any $X,Y,Z\in TM$, where $\xi $ is the Reeb vector field, $\langle \cdot
,\cdot \rangle =g_{\theta }(\cdot ,\cdot )$ and%
\begin{equation}
S(X,Y)=(\nabla _{X}\tau )Y-(\nabla _{Y}\tau )X,\text{ }\mathcal{O}=\tau
^{2}+2J\tau -I,\text{ }L=\tau +J.  \label{equation5.7}
\end{equation}%
Using either \eqref{equation5.6} or the last sum term in \eqref{equation2.16}, we can deduce that: 
\begin{equation}
\langle R(\xi ,Y)Z,W\rangle =\langle S(Z,W),Y\rangle  \label{equation5.8}
\end{equation}%
for any $Y,Z,W\in H$. In terms of \eqref{equation5.6}, \eqref{equation5.7}, \eqref{equation5.8} and the assumptions
for $K_{\sec }^{H}$ and the torsion $\tau $, it is easy to find that the
Riemannian sectional curvature of $\nabla ^{\theta }$ is bounded below by a
constant $-\Lambda ^{2}$ depending only on $\delta _{1}$ and $\delta _{2}$.

By the definitions for Hessians with respect to $\nabla ^{\theta
}\allowbreak $ and $\nabla $ respectively, we have for any $X,Y\in H$ that 
\begin{equation}
\begin{aligned}
(\nabla ^{\theta }dr^{2})(X,Y)-(\nabla dr^{2})(X,Y) & =-(\nabla _{X}^{\theta
}Y-\nabla _{X}Y)r^{2} \\ 
& ={\big(}d\theta (X,Y)+A(X,Y){\big)}\xi (r^{2}).%
\end{aligned}%
\label{equation5.9}
\end{equation}%
Let $\{\eta _{k}\}$ be any unitary basis in $T_{1,0}M$. Taking $X=\eta _{k}$%
, $Y=\eta _{\overline{l}}$, it follows from \eqref{equation5.9} that%
\begin{equation*}
(\nabla ^{\theta }dr^{2})(\eta _{k},\eta _{\overline{l}})=(\nabla
dr^{2})(\eta _{k},\eta _{\overline{l}})+i\delta _{kl}\xi (r^{2}).
\end{equation*}%
Consequently%
\begin{equation}
\sum_{k,l} \left[ (r^{2})_{k\overline{l}}\zeta _{k}\zeta _{\overline{l}
}+(r^{2})_{\overline{k}l}\zeta _{\overline{k}}\zeta
_{l} \right] =\sum_{k,l} \left[ (\nabla ^{\theta }dr^{2})(\eta _{k},\eta _{\overline{l}
})\zeta _{k}\zeta _{\overline{l}}+(\nabla ^{\theta }dr^{2})(\eta _{\overline{%
k}},\eta _{l})\zeta _{\overline{k}}\zeta _{l} \right] 
\label{equation5.10}
\end{equation}%
where $(r^{2})_{k\overline{l}}=(\nabla dr^{2})(\eta _{k},\eta _{\overline{l}%
})$ and $(\zeta _{1},...,\zeta _{m})\ $is any unitary vector in $C^{m}$.
Combining \eqref{equation5.10} and Lemma \ref{lemma5.4}, we conclude that there is a constant $%
D=D(m,\delta _{1},\delta _{2})$ such that 
\begin{equation*}
\sum_{k,l}\{(r^{2})_{k\overline{l}}\zeta _{k}\zeta _{\overline{l}}+(r^{2})_{%
\overline{k}l}\zeta _{\overline{k}}\zeta _{l}\}\leq D(1+r).
\end{equation*}%
\end{proof}

The following result improves the upper bound of $f^{\ast }G_{\widetilde{%
\theta }}$ in Theorem \ref{theorem5.2} for a transversally holomorphic map.

\begin{theorem} \label{theorem5.6}
Let $(M^{2m+1},H,J,\theta )$ be a complete
pseudo-Hermitian manifold with pseudo-Hermitian sectional curvature bounded
from below by $-k_{1}$ ($k_{1}\geq 0$) and $\Vert A\Vert _{C^{1}}$ bounded
from above. Let $(N^{2n+1},\widetilde{H},\widetilde{J},\widetilde{\theta })$
be a Sasakian manifold with pseudo-Hermitian sectional curvature bounded
from above by $-k_{2}<0$. If the horizontal sectional curvature $K_{\sec
}^{H}$ of $M$ is bounded from below, then any transversally holomorphic map $%
f:M\rightarrow N$ satisfies:%
\begin{equation*}
f^{\ast }\widetilde{G}_{\theta }\leq \frac{k_{1}}{k_{2}}G_{\theta }.
\end{equation*}%
In particular, if $k_{1}=0$, then every transversally holomorphic map is
horizontally constant.
\end{theorem}

\begin{proof}
At each point $x\in M$, we may choose a unitary frame $%
\{\zeta _{i}(x)\}_{i=1}^{m}$ for $T_{1,0}M$ such that 
\begin{equation*}
f^{\ast }G_{\widetilde{\theta }}=\sum_{i=1}^{m}\lambda _{i}(x)\theta
^{i}\theta ^{\overline{i}}.
\end{equation*}%
where $\lambda _{1}(x)\geq \lambda _{2}(x)\geq \ldots \geq \lambda _{m}(x)$
are the eigenvalues of $\left( \sum_{\alpha }f_{i}^{\alpha }f_{\overline{j}%
}^{\overline{\alpha }}\right) _{x}$. We also write the largest eigenvalue $%
\lambda _{1}(x)$ as $\lambda (x)$ for simplicity. Clearly we have 
\begin{equation}
f^{\ast }G_{\widetilde{\theta }}\leq \lambda (x)G_{\theta }.  \label{equation5.11}
\end{equation}%
Thus it suffices to estimate the supremum of of the function $\lambda $. Let 
$r$ denote the Riemannian distance function of $g_{\theta }$ relative to a
fixed point $p$. Define 
\begin{equation}
\phi (x)=\left( a^{2}-r^{2}(x)\right) ^{2}\lambda (x)  \label{equation5.12}
\end{equation}%
for $x\in B_{p}(a)$. Since $M$ is complete and $\phi \mid _{\partial
B_{p}(a)}=0$, we see that $\phi $ attains its maximum at some point $x_{0}$
in $B_{p}(a)$. Without loss of generality, we assume that $\lambda
\not\equiv 0$ on $B_{p}(a)$, so $\lambda (x_{0})>0$. Note that both $\lambda 
$ and $r$ may not be differentiable near $x_{0}$. One may remedy the
non-differentiability of $\lambda $ by the following method. Choose any
local (smooth) unitary frame field $\{\eta _{i}\}$ for $T_{1,0}M$ such that $%
\eta _{1}(x_{0})=\zeta _{1}(x_{0})$ is the eigenvector of $\lambda (x_{0})$.
Define%
\begin{equation}
\widetilde{\lambda }=\mid df(\eta _{1})\mid ^{2}=\sum_{\alpha }f_{1}^{\alpha
}f_{\overline{1}}^{\overline{\alpha }}  \label{equation5.13}
\end{equation}%
which is obviously smooth near $x_{0}$, and%
\begin{equation}
\widetilde{\lambda }(x)\leq \lambda (x) , \qquad 
\widetilde{\lambda }(x_{0})=\lambda (x_{0}).
\label{equation5.14}
\end{equation}%
Hence $\left( a^{2}-r^{2}(x)\right) ^{2}\widetilde{\lambda }(x)$ also
attains its (local) maximum at $x_{0}$ and 
\begin{equation*}
\left( a^{2}-r^{2}(x_{0})\right) ^{2}\widetilde{\lambda }(x_{0})=\phi
(x_{0}).
\end{equation*}%
The possible non-differentiability of $r$ can be treated as usual (cf. \cite{[Che]}). Without loss of generality, we may assume that $r$ is already
smooth near $x_{0}$. Let%
\begin{equation}
\widetilde{\phi }(x)=\left( a^{2}-r^{2}(x)\right) ^{2}\widetilde{\lambda }%
(x).  \label{equation5.15}
\end{equation}%
By computing the gradient and Hessian of $\log \widetilde{\phi }(x)$ at the
maximum point $x_{0}$, we have%
\begin{align}
\frac{\nabla \widetilde{\lambda }}{\widetilde{\lambda }}+2\frac{\nabla
\left( a^{2}-r^{2}\right) }{a^{2}-r^{2}} & =0,  \label{equation5.16} \\
(\log \widetilde{\phi })_{1\overline{1}}+(\log \widetilde{\phi })_{\overline{%
1}1} & \leq 0,  \label{equation5.17}
\end{align}%
where $(\log \widetilde{\phi })_{i\overline{j}}$ denotes the Hessian defined
with respect to the Tanaka-Webster connection. It follows from \eqref{equation5.17} and
\eqref{equation5.16} that 
\begin{equation}
\frac{\widetilde{\lambda }_{1\overline{1}}+\widetilde{\lambda }_{\overline{1}%
1}}{\widetilde{\lambda }}-12\frac{\mid (a^{2}-r^{2})_{1}\mid ^{2}}{%
(a^{2}-r^{2})^{2}}-2\frac{(r^{2})_{1\overline{1}}+(r^{2})_{\overline{1}1}}{%
a^{2}-r^{2}}\leq 0.  \label{equation5.18}
\end{equation}%
holds at $x_{0}\in B_{p}(a)$. Clearly 
\begin{equation}
\mid (a^{2}-r^{2})_{1}(x_{0})\mid^2 \leq 4a^{2}.  \label{equation5.19}
\end{equation}%
In terms of Lemma \ref{lemma5.5} by taking $\zeta =(1,0,..,0)$, we obtain%
\begin{equation}
(r^{2})_{1\overline{1}}(x_{0})+(r^{2})_{\overline{1}1}(x_{0})\leq D(1+a). 
\label{equation5.20}
\end{equation}%
Since $f$ is transversally holomorphic and $N$ is Sasakian, it follows from
\eqref{equation3.11} and \eqref{equation3.14} that%
\begin{equation}
f_{11\overline{1}}^{\alpha }=\sum_{i}f_{i}^{\alpha }R_{11\overline{1}%
}^{i}-\sum_{\beta ,\gamma ,\delta }f_{1}^{\beta }f_{1}^{\gamma }f_{\overline{%
1}}^{\overline{\delta }}\widetilde{R}_{\beta \gamma \overline{\delta }%
}^{\alpha }.  \label{equation5.21}
\end{equation}%
Due to \eqref{equation5.13} and \eqref{equation3.11}, we get%
\begin{equation}
\widetilde{\lambda }_{1\overline{1}}+\widetilde{\lambda }_{\overline{1}%
1}=2\sum_{\alpha }{\big(}\mid f_{11}^{\alpha }\mid ^{2}+f_{\overline{1}}^{%
\overline{\alpha }}f_{11\overline{1}}^{\alpha }+f_{1}^{\alpha }f_{\overline{1%
}\overline{1}1}^{\overline{\alpha }}{\big).}  \label{equation5.22}
\end{equation}%
Substituting \eqref{equation5.21} into \eqref{equation5.22} and using the curvature assumptions for both 
$M$ and $N$, we find that%
\begin{align}
\left( \widetilde{\lambda }_{1\overline{1}}+\widetilde{\lambda }_{\overline{1%
}1}\right) (x_{0}) & \geq \sum_{i,\alpha }(f_{i}^{\alpha }f_{\overline{1}}^{\overline{\alpha }%
}R_{11\overline{1}}^{i}+f_{\overline{i}}^{\overline{\alpha }}f_{1}^{\alpha
}R_{\overline{1}\overline{1}1}^{\overline{i}})-\sum_{\alpha ,\beta ,\gamma
,\delta }(f_{\overline{1}}^{\overline{\alpha }}f_{1}^{\beta }f_{1}^{\gamma
}f_{\overline{1}}^{\overline{\delta }}\widetilde{R}_{\beta \gamma \overline{%
\delta }}^{\alpha }+f_{1}^{\alpha }f_{\overline{1}}^{\overline{\beta }}f_{%
\overline{1}}^{\overline{\gamma }}f_{1}^{\delta }\widetilde{R}_{\overline{%
\beta }\overline{\gamma }\delta }^{\overline{\alpha }}) \nonumber \\ 
& \geq -2k_{1}\lambda (x_{0})+2k_{2}\lambda ^{2}(x_{0}).
\label{equation5.23}
\end{align}
From \eqref{equation5.18}, \eqref{equation5.19}, \eqref{equation5.20} and \eqref{equation5.23}, we may deduce that 
\begin{equation}
\lambda (x_{0})\leq \frac{k_{1}}{k_{2}}+\frac{24a^{2}}{%
k_{2}(a^{2}-r^{2}(x_{0}))^{2}}+\frac{D(1+a)}{k_{2}(a^{2}-r^{2}(x_{0}))}. 
\label{equation5.24}
\end{equation}%
Consequently%
\begin{equation*}
(a^{2}-r^{2}(x))^{2}\lambda (x)\leq (a^{2}-r^{2}(x_{0}))^{2}\lambda (x_{0})
\leq \frac{k_{1}}{k_{2}}a^{4}+\frac{24a^{2}}{k_{2}}+\frac{D(1+a)a^{2}}{k_{2}}
\end{equation*}%
which implies%
\begin{equation}
\lambda (x)\leq \frac{k_{1}}{k_{2}}\frac{a^{4}}{(a^{2}-r^{2}(x))^{2}}+\frac{%
24a^{2}}{k_{2}(a^{2}-r^{2}(x))^{2}}+\frac{D(1+a)a^{2}}{(a^{2}-r^{2}(x))^{2}}
\label{equation5.25}
\end{equation}%
for any $x\in B_{p}(a)$. Letting $a\rightarrow \infty $ in \eqref{equation5.25}, we get%
\begin{equation*}
\sup_{M}\lambda \leq \frac{k_{1}}{k_{2}}.
\end{equation*}%
\end{proof}

Note that if $m=1$ in Theorem \ref{theorem5.6}, the horizontal sectional curvature is
just the pseudo-Hermitian sectional curvature. Therefore we have

\begin{corollary} \label{corollary5.7}
Let $(M^{3},H,J,\theta )$ be a complete $3$%
-dimensional pseudo-Hermitian manifold with pseudo-Hermitian sectional
curvature bounded from below by $-k_{1}$ ($k_{1}\geq 0$) and $\Vert A\Vert
_{C^{1}}$ bounded from above. Let $(N^{2n+1},\widetilde{H},\widetilde{J},%
\widetilde{\theta })$ be a Sasakian manifold with pseudo-Hermitian sectional
curvature bounded from above by $-k_{2}<0$. Then any transversally
holomorphic map $f:M\rightarrow N$ satisfies:%
\begin{equation*}
f^{\ast }\widetilde{G}_{\theta }\leq \frac{k_{1}}{k_{2}}G_{\theta }.
\end{equation*}%
In particular, if $k_{1}=0$, then every transversally holomorphic map is
horizontally constant.
\end{corollary}

Suppose $f:M\rightarrow N$ is a smooth map between two smooth manifolds of
the same dimension $2m+1$. We say that $f$ is degenerate at $x\in M$ if the
Jacobian of $f$ vanishes at $x$. Geometrically this means that $%
df_{x}:T_{x}M\rightarrow T_{f(x)}N$ is not univalent. The map $f$ is called
totally degenerate if its Jacobian vanishes identically.

\begin{theorem} \label{theorem5.8}
Let $(M,H,J,\theta )$ and $(N,\widetilde{H},\widetilde{%
J},\widetilde{\theta })$ be pseudo-Hermitian manifolds of the same dimension 
$2m+1$, with $M$ compact without boundary and $N$ Sasakian. Let $R_{M}$ be
the pseudo-Hermitian scalar curvature of $M$ and $Ric_{N}$ be the
pseudo-Hermitian Ricci curvature of $N$. Then for any transversally
holomorphic map $f:M\rightarrow N$, we have
\begin{enumerate}[(1)]
  \item If $R_{M}>0$, $Ric_{N}\leq 0$, then $f$ is totally degenerate;
  \item If $R_{M}<0$, $Ric_{N}\geq 0$, then there is a point $x_{0}$ of $M$ at
which $f$ is degenerate.
\end{enumerate}
\end{theorem}

\begin{proof}
Clearly $\det \left( df\right) =f_{0}^{0}\Phi \overline{\Phi 
}$, where $\Phi =\det \left( f_{i}^{\alpha }\right) $. Set $v=\Phi \overline{%
\Phi }$. Since $M$ is compact, $v$ attains its maximum at a point $x_{0}\in
M $. If $v$ is not identically zeor, then $v(x_{0})>0$. At $x_{0}$, we have%
\begin{equation}
(v_{k})_{x_{0}}=0\text{ for }k=1,...,m,\text{ and }(\bigtriangleup
_{b}v)_{x_{0}}\leq 0.  \label{equation5.26}
\end{equation}%
Under the conditions in (1), we have from \eqref{equation3.41}, \eqref{equation3.35} and \eqref{equation5.26} that 
\begin{equation*}
\frac{1}{2}(\bigtriangleup _{b}v)_{x_{0}}=v(x_{0}) \left( R-\sum_{k,\gamma
,\delta }\widetilde{R}_{\gamma \overline{\delta }}f_{k}^{\gamma }f_{%
\overline{k}}^{\overline{\delta }}\right) > 0.
\end{equation*}%
This gives a contradiction. Thus (1) is proved. Similarly (2) can be proved
by the consideration of the minimum of $v$.
\end{proof}

Finally in this section, we would like to give a Schwarz type lemma for the
ratio of horizontal volume elements under a transversally holomorphic map.

\begin{theorem} \label{theorem5.9}
Let $(M^{2m+1},H,J,\theta )$ be a complete
pseudo-Hermitian manifold with pseudo-Hermitian Ricci curvature bounded from
below by $-mk_{1}$ ($k_{1}\geq 0$) and $\Vert A\Vert _{C^{1}\text{ }}$%
bounded from above. Let $(N,\widetilde{H},\widetilde{J},\widetilde{\theta })$
be a Sasakian manifold with the pseudo-Hermitian Ricci curvature bounded
from above by $-mk_{2}<0$ and the same dimension as $M$. Then for any
transversally holomorphic map $f:M\rightarrow N$, we have%
\begin{equation*}
f^{\ast }\Omega _{N}^{\widetilde{H}}\leq \left( \frac{k_{1}}{k_{2}}\right)
^{m}\Omega _{M}^{H}
\end{equation*}%
where $\Omega _{M}^{H}$ and $\Omega _{N}^{H}$ are the horizontal volume
elements of $M$ and $N$ respectively. In particular, if $k_{1}=0$, then $f$
is totally degenerate.
\end{theorem}

\begin{proof}
We have already known that $f^{\ast }\Omega _{N}^{\widetilde{%
H}}=\Phi \overline{\Phi }\Omega _{M}^{H}$, where $\Phi =\det (f_{i}^{\alpha
})$. Set $v=\Phi \overline{\Phi }$. According to \eqref{equation3.41} and the curvature
assumptions, we have%
\begin{equation}
\frac{1}{2}\frac{\bigtriangleup _{b}v}{v}\geq -m^{2}k_{1}+mk_{2}\sum_{\alpha
,k}f_{k}^{\alpha }f_{\overline{k}}^{\overline{\alpha }} 
\geq -m^{2}k_{1}+m^{2}k_{2}v^{\frac{1}{m}},%
\label{equation5.27}
\end{equation}%
where the second inequality follows from the Hadamard's determinant
inequality for a $m\times m$ matrix, that is,%
\begin{equation*}
\frac{1}{m}\sum_{i,k}\mid a_{ik}\mid ^{2}\geq \mid \det (a_{ik})\mid ^{2/m}.
\end{equation*}%
Now we consider the following function on $B_{p}(a)$:%
\begin{equation}
\phi (x)=\left( a^{2}-r^{2}(x)\right) ^{2m}v(x)  \label{equation5.28}
\end{equation}%
where $r(x)$ is the Riemannian distance of $g_{\theta }$ relative to $p$. By
a similar argument for $\phi $ as in Theorems \ref{theorem4.2} and \ref{theorem5.2}, we may deduce that%
\begin{equation*}
\sup_{M}v\leq \left( \frac{k_{1}}{k_{2}}\right) ^{m}\text{.}
\end{equation*}%
Then this theorem is proved.
\end{proof}

\section{Pseudodistances and hyperbolicity} \label{section6}

We still consider strictly pseudoconvex \emph{CR} manifolds in this section,
although some notions and results in following hold true for more general 
\emph{CR} manifolds. Let $D^{3}(-1)$ denote a $3$-dimensional Sasakian space
form with pseudo-Hermitian curvature $-1$. We denote by $\rho _{CC}$ the
Carnot--Carath\'{e}odory distance of $D^{3}(-1)$. Now we may define a
pseudodistance $d_{M}^{K}$ for a strictly pseudoconvex \emph{CR} manifold $M$
as follows: For any two points $p,q\in M$, by a \emph{CR K-chain} between
them, we mean three groups of data: ${\big\{}\{p_{0},p_{1},...,p_{k}\},%
\{a_{1}$, $b_{1},...,a_{k},b_{k}\},\{f_{1},...,f_{k}\}{\big\}}$ ($%
k<\infty $), where $\{p_{0},p_{1},...,p_{k}\}$ is a finite set of points in $%
M$ with $p_{0}=p$ and $p_{k}=q$, $\{a_{1},b_{1},...,a_{k},b_{k}\}$ is a set
of points in $D^{3}(-1)$, and $\{f_{1},...,f_{k}\}$ is a set of \emph{CR}
maps of $D^{3}(-1)$ into $M$ such that $f_{i}(a_{i})=p_{i-1}$ and $%
f_{i}(b_{i})=p_{i}$ for $i=1,...,k$. Then we define a \emph{CR
K-pseduodistance} by 
\begin{equation}
d_{M}^{K}(p,q)=\inf \{\rho _{CC}(a_{1},b_{1})+\cdots +\rho
_{CC}(a_{k},b_{k})\}  \label{equation6.1}
\end{equation}%
where the infimum is taken for all possible \emph{CR }K-chains between $p$
and $q$. It is an easy matter to verify that $d_{M}^{K}:M\times M\rightarrow
R$ is continuous and satisfies the axioms for pseudodistance:%
\begin{equation}
d_{M}^{K}(p,q)\geq 0\text{, \ }d_{M}^{K}(p,q)=d_{M}^{K}(q,p)\text{, \ }%
d_{M}^{K}(p,q)+d_{M}^{K}(q,s)\geq d_{M}^{K}(p,s).  \label{equation6.2}
\end{equation}%
The following proposition follows immediately from the definition for $d^{K}$%
.

\begin{proposition} \label{proposition6.1}
Let $M$ and $N$ be two strictly pseudoconvex \emph{%
CR} manifolds. If $f:M\rightarrow N$ is a \emph{CR} map, then 
\begin{equation*}
d_{M}^{K}(p,q)\geq d_{N}^{K}(f(p),f(q)),
\end{equation*}%
Furthermore, if $f$ is a \emph{CR} diffeomorphism, then it is an isometry
with repect to the \emph{CR }K-pseudodistances, i.e., $%
d_{M}^{K}(p,q)=d_{N}^{K}(f(p),f(q))$ \ \ for $p,q\in M$.
\end{proposition}

Considering our Schwarz tpye results about \emph{CR} maps with bounded
generalized dilatation in \S \ref{section4}, we would like to introduce a somewhat
restricted pseudodistance as follows. If $f:D^{3}(-1)\rightarrow M$ is a
nontrivial \emph{CR} map of bounded generalized dilatation $\Lambda $, then
we see from \eqref{equation4.4} that $\Lambda \geq 1/\sqrt{2}$. Therefore let us fixed a
positive nubmer $\Lambda \in \lbrack 1/\sqrt{2},\infty )$. We say that a 
\emph{CR }K-chain ${\big\{}\{p_{0},p_{1},...,p_{k}\},%
\{a_{1},b_{1},...,a_{k},b_{k}\}$, $\{f_{1},...,f_{k}\}{\big\}}$ is of
order $\Lambda $ if each $f_{i}$ ($i=1,...,k$) is a \emph{CR} map with
bounded generalized dilatation of order $\Lambda $. Replacing \emph{CR }%
K-chains in \eqref{equation6.1} by \emph{CR }K-chains of order $\Lambda $, one may
introduce a \emph{CR K-pseudodistance} of order $\Lambda $ by%
\begin{equation}
d_{M}^{K,\Lambda }(p,q)=\inf \{\rho _{CC}(a_{1},b_{1})+\cdots +\rho
_{CC}(a_{k},b_{k})\}  \label{equation6.3}
\end{equation}%
where the infimum is taken for all possible \emph{CR }K-chains of order $%
\Lambda $ between $p$ and $q$. Clearly $d_{M}^{K,\Lambda }$ also satisfies
the properties in \eqref{equation6.2}, and if $0<$ $\Lambda <\widetilde{\Lambda }$, then $%
d_{M}^{K,\Lambda }(p,q)\geq d_{M}^{K,\widetilde{\Lambda }}(p,q)\geq
d_{M}^{K}(p,q)$ for any $p,q\in M$. We don't know whether any two points of
a general pseudo-Hermitian manifold can be joint by a \emph{CR} chain (resp. 
\emph{CR K}-chain of order $\Lambda $). However, some special
pseudo-Hermitian manifolds, e.g., compact spherical pseudo-Hermitian
manifolds should have this property. Whenever any two points of $M$ can be
joint by a \emph{CR} chain (resp. a \emph{CR K}-chain of order $\Lambda $),
we will say that $d_{M}^{K}$ (resp. $d_{M}^{K,\Lambda }$) is well-defined.

\begin{definition} \label{definition6.1}
Let $M$ be a strictly pseudo-convex \emph{CR}
manifold. If $d_{M}^{K}$ is a well-defined distance, i.e., $d_{M}^{K}(p,q)>0$
for $p\neq q$, then $M$ is called a \emph{CR} Kobayashi hyperbolic manifold.
If $d_{M}^{K,\Lambda }$ is a well-defined distance for some positive number $%
\Lambda $, then $M$ is called a \emph{CR} Kobayashi hyperbolic manifold of
order $\Lambda $.
\end{definition}

According to Proposition \ref{proposition6.1}, the \emph{CR} Kobayashi hyperbolicity is
invariant under a \emph{CR} diffeomorphism. At the present status, we don't
know any example of \emph{CR} Kobayashi hyperbolic manifolds. However we
have the following

\begin{theorem} \label{theorem6.2}
Let $(M^{2m+1},H,J,\theta )$ be a pseudo-Hermitian
manifold in which any two points can be joint by a \emph{CR} K-chain of
order $\Lambda $. Suppose its pseudo-Hermitian sectional curvature is
bounded from above by a negative constant $-k$. Then $M$ is a \emph{CR}
Kobayashi hyperbolic manifold of order $\Lambda $.
\end{theorem}

\begin{proof}
Let $L^{D}$ denote the Levi metric on $D^{3}(-1)$ and let $%
L_{\theta }^{M}$ be the Levi metric of $M$. We find by Corollary \ref{corollary4.3} that
the following inequality holds 
\begin{equation}
f_{H,\widetilde{H}}^{\ast }L_{\theta }^{M}\leq \frac{8\Lambda ^{2}+1}{k}L^{D}
\label{equation6.4}
\end{equation}%
for every \emph{CR} map $f:D^{3}(-1)\rightarrow M$ which has bounded
generalized dilatation of order $\Lambda $. If we denote by $r_{CC}^{M}$ and 
$\rho _{CC}$ the \emph{CC}-distances on $M^{2m+1}$ and $D^{3}$ defined by $%
L_{\theta }^{M}$ and $L^{D}$ respectively, then \eqref{equation6.4} implies that $f$ is
distance-decreasing (up to a constant) with respect to $r_{CC}^{M}$ and $%
\rho _{CC}$, that is, 
\begin{equation}
r_{CC}^{M}(f(a),f(b))\leq \frac{8\Lambda ^{2}+1}{k}\rho _{CC}(a,b)  \label{equation6.5}
\end{equation}%
for any $a,b\in D^{3}(-1)$. Consequently we have 
\begin{equation*}
d_{M}^{K,\Lambda }(p,q)\geq \frac{k}{8\Lambda ^{2}+1}r_{CC}^{M}(p,q)>0
\end{equation*}%
if $p\neq q$.
\end{proof}

\begin{proposition} \label{proposition6.3}
For $D^{3}(-1)$, $d_{D}^{K,\Lambda }$ is
equivalent to $\rho _{CC}$ for any $\Lambda \geq \frac{\sqrt{2}}{2}$. In
particular, $D^{3}(-1)$ is a \emph{CR} Kobayashi hyperbolic manifold of
order $\Lambda $ for each $\Lambda \in \lbrack \frac{\sqrt{2}}{2},\infty )$.
\end{proposition}

\begin{proof}
Let $L^{D}$ denote the Levi metric on $D^{3}(-1)$. If $%
f:D^{3}(-1)\rightarrow D^{3}(-1)$ is a \emph{CR} map with bounded
generalized dilatation of order $\Lambda $, we have already known in
Corollary \ref{corollary4.3} that%
\begin{equation*}
f_{H,H}^{\ast }L^{D}\leq (8\Lambda ^{2}+1)L^{D}
\end{equation*}%
which implies that 
\begin{equation}
d_{D^{3}(-1)}^{K,\Lambda }(p,q)\geq \frac{1}{(8\Lambda ^{2}+1)}\rho
_{CC}(p,q)  \label{equation6.6}
\end{equation}%
for any $p,q\in D^{3}(-1)$. Clearly the identity transformation $%
id:D^{3}(-1)\rightarrow D^{3}(-1)$ is a \emph{CR} map with generalized
dilatation $1/\sqrt{2}$, so it has bounded generalized dilatation of order $%
\Lambda $ for any $\Lambda \geq 1/\sqrt{2}$. Therefore $d_{D^{3}(-1)}^{K,%
\Lambda }(p,q)\leq \rho _{CC}(p,q)$. We conclude that $d_{D^{3}(-1)}^{K,%
\Lambda }(p,q)$ and $\rho _{CC}$ are equivalent pseudodistances.
\end{proof}

In order to study some invariance of the pseudodistance $d^{K,\Lambda }$, we
recall the notion of quasi-isometry as follows. A Riemannian manifold $(M,g)$
is said to be \emph{quasi-isometric} to another Riemannian manifold $(N,%
\widetilde{g})$ if there is a diffeomorphism $\phi :M\rightarrow N$ and a
constant $C>0$ such that 
\begin{equation*}
C^{-1}g\leq \phi ^{\ast }\widetilde{g}\leq Cg.
\end{equation*}%
Clearly $\phi ^{-1}$ is also a quasi-isometry with the same coefficient of
expansion. Now we say that two pseudo-Hermitian manifolds $(M,H,J,\theta )$
and $(N,\widetilde{H},\widetilde{J},\widetilde{\theta })$ are \emph{%
CR-quasi-isometric} if there is a \emph{CR} diffeomorphism $\phi
:M\rightarrow N$ which is quasi-isometric with repect to the Webster metrics
on $M$ and $N$.

\begin{example} \label{example6.1}
Let $(M,H,J,\theta )$ be a pseudo-Hermitian manifold.
Set $\widetilde{\theta }=e^{\varphi }\theta $ for some $\varphi \in
C^{\infty }(M)$ (or equivalently, $\theta =e^{-\varphi }\widetilde{\theta }$%
). Then 
\begin{align*}
g_{\widetilde{\theta }} =\widetilde{\theta }\otimes \widetilde{\theta }+G_{%
\widetilde{\theta }} = e^{2\varphi }\theta \otimes \theta (\cdot ,\cdot )+e^{\varphi }(d\varphi
\wedge \theta )(\cdot ,J\cdot )+e^{\varphi }d\theta (\cdot ,J\cdot ).
\end{align*}%
Writing $d\varphi =\varphi _{0}\theta +\varphi _{k}\theta ^{k}+\varphi _{%
\overline{k}}\theta ^{\overline{k}}$ in terms of the coframe field in \S \ref{section3},
it is easy to derive that if $\parallel \varphi \parallel _{C^{1}}$ is
bounded from above, then there exists a constant $C>1$ such that $g_{%
\widetilde{\theta }}\leq Cg_{\theta }$, and thus $g_{\theta }\leq Cg_{%
\widetilde{\theta }}$ too. Hence we find that $id:(M,H,J,\theta )\rightarrow
(M,H,J,\widetilde{\theta })$ is a \emph{CR} quasi-isometry.
\end{example}

\begin{lemma} \label{lemma6.4}
Let $(U^{n},g_{U})$, $(V^{m},g_{V})$ and $(W^{m},g_{W})\ 
$be Euclidean vector spaces. Let $A:U\rightarrow V$ be a linear map with
generalized dilatation of order $\Lambda $. Suppose $B:V\rightarrow W$ is a
linear isomorphism, and $\mu _{1}\geq \mu _{2}\geq \cdots \geq \mu _{m}>0$
are the eigenvalues of $B^{t}B$. Then $B\circ A:V\rightarrow W$ is a linear
map with generalized dilatation of order $\sqrt{\frac{\mu _{1}}{\mu _{m}}}%
\Lambda $.
\end{lemma}

\begin{proof}
Let $\lambda _{1}\geq \lambda _{2}\geq \cdots \geq \lambda
_{n}\geq 0$ be the eigenvalues of $A^{t}A$ and let $\gamma _{1}\geq \gamma
_{2}\geq \cdots \geq \gamma _{n}$ be the eigenvalues of $\left( BA\right)
^{t}(BA)$. Let $u_{i}$ ($1\leq i\leq n$) be an orthonormal basis in $U$ such
that $u_{i}$ is an eigenvector of $\left( BA\right)^{t}(BA)$ corresponding to $\gamma _{i}$ ($i=1,2,...,n$%
). Clearly we have%
\begin{equation}
\gamma _{1}=\max_{\Vert u\Vert =1}\parallel BAu\parallel ^{2}\leq \mu
_{1}\lambda _{1}  \label{equation6.7}
\end{equation}%
and%
\begin{equation}
\begin{aligned}
\gamma _{2}+\cdots +\gamma _{n}
& =\sum_{k=2}^{n}\langle \left( BA\right)^{t}\circ \left( BA\right) u_{k},u_{k}\rangle =\sum_{k=2}^{n}\parallel BAu_{k}\parallel ^{2} \\ 
& \geq \mu _{m}\sum_{k=2}^{n}\parallel Au_{k}\parallel ^{2} =\mu _{m}\left( \sum_{k=1}^{n}\parallel Au_{k}\parallel ^{2}-\parallel
Au_{1}\parallel ^{2}\right) \\ 
& \geq \mu _{m}\left( \lambda _{2}+\cdots +\lambda _{m}\right)%
\end{aligned}%
\label{equation6.8}
\end{equation}%
The assumption for $A$ yields%
\begin{equation}
\lambda _{2}+\cdots +\lambda _{m}\geq \frac{1}{\Lambda ^{2}}\lambda _{1}. 
\label{equation6.9}
\end{equation}%
It follows from \eqref{equation6.7}, \eqref{equation6.8} and \eqref{equation6.9} that%
\begin{equation*}
\gamma _{1}\leq \frac{\mu _{1}\Lambda ^{2}}{\mu _{m}}\left( \gamma
_{2}+\cdots +\gamma _{n}\right) .
\end{equation*}%
This shows that $B\circ A:U\rightarrow W$ has the generalized dilatation of
order $\sqrt{\frac{\mu _{1}}{\mu _{m}}}\Lambda $.
\end{proof}

In terms of Lemma \ref{lemma6.4}, we immediately get

\begin{proposition} \label{proposition6.5}
Let $f:(M,H,J,\theta )\rightarrow (N,\widetilde{H},%
\widetilde{J},\widetilde{\theta })$ be a \emph{CR} map with bounded
generalized dilatation of order $\Lambda $. Suppose $h:(N,\widetilde{H},%
\widetilde{J},\widetilde{\theta })\rightarrow (W,\widehat{H},\widehat{J},%
\widehat{\theta })$ is a \emph{CR}-quasi-isometry satisfying $C^{-1}g_{%
\widetilde{\theta }}\leq h^{\ast }g_{\widehat{\theta }}\leq Cg_{\widetilde{%
\theta }}$ for some constant $C\geq 1$. Then $h\circ f$ is a \emph{CR} map
with bounded generalized dilatation of order $C\Lambda $.
\end{proposition}

\begin{theorem} \label{theorem6.6}
Let $h:(M,H,J,\theta )\rightarrow (N,\widetilde{H},%
\widetilde{J},\widetilde{\theta })$ be a \emph{CR}-quasi-isometry satisfying 
$C^{-1}g_{\widetilde{\theta }}\leq h^{\ast }g_{\widehat{\theta }}\leq Cg_{%
\widetilde{\theta }}$ for some constant $C\,\geq 1$. Suppose $%
d_{M}^{K,C\Lambda }$, $d_{M}^{K,\Lambda /C}$ and $d_{N}^{K,\Lambda }$ are
well-defined for some $\Lambda >0$. Then%
\begin{equation*}
d_{M}^{K,C\Lambda }(p,q)\leq d_{N}^{K,\Lambda }(h(p),h(q))\leq
d_{M}^{K,\Lambda /C}(p,q)
\end{equation*}%
for any $p,q\in M$.
\end{theorem}

\begin{proof}
Suppose ${\big\{}\{p_{0},p_{1},...,p_{k}\},%
\{a_{1},b_{1},...,a_{k},b_{k}\}$, $\{f_{1},...,f_{k}\}{\big\}}$ is a 
\emph{CR} K-chain of order $\Lambda /C$ joining $p$ and $q$ in $M$.
According to Proposition \ref{proposition6.5}, we find that 
\begin{align*}
{\big\{}%
\{h(p_{0}),h(p_{1}),...,h(p_{k})\},\{a_{1},b_{1},...,a_{k},b_{k}\},
\{h\circ f_{1},...,h\circ f_{k}\}{\big\}}
\end{align*} 
is a \emph{CR} K-chain of
order $\Lambda $ joining $h(p)$ and $h(q)$ in $N$. By the definition of $%
d^{K,\Lambda }$, we have%
\begin{equation}
d_{N}^{K,\Lambda }(h(p),h(q))\leq d_{M}^{K,\Lambda /C}(p,q)\text{.} 
\label{equation6.10}
\end{equation}%
Since $h^{-1}$ is also a \emph{CR} quasi-isometry with the same coefficient
of expansion, we get the following inequality by applying \eqref{equation6.10} to $h^{-1}$%
\begin{equation}
d_{M}^{K,C\Lambda }(p,q)\leq d_{N}^{K,\Lambda }(h(p),h(q)).  \label{equation6.11}
\end{equation}%
This lemma follows from \eqref{equation6.10} and \eqref{equation6.11}. 
\end{proof}

\begin{corollary} \label{corollary6.7}
Let $h:(M,H,J,\theta )\rightarrow (N,\widetilde{H},%
\widetilde{J},\widetilde{\theta })$ be a \emph{CR}-quasi-isometry satisfying 
$C^{-1}g_{\widetilde{\theta }}\leq h^{\ast }g_{\widehat{\theta }}\leq Cg_{%
\widetilde{\theta }}$ for some constant $C\geq 1$. Suppose $M$ is a \emph{CR}
Kobayashi hyperbolic manifold of order $C\Lambda $ and $d_{N}^{K,\Lambda }$
is well-defined. Then $(N,\widetilde{H},\widetilde{J},\widetilde{\theta })$
is a \emph{CR} Kobayashi hyperbolic manifold of order $\Lambda $. In
particular, if there exists a positive number $\Lambda _{0}$ such that $M$
is a \emph{CR} Kobayashi hyperbolic manifold of order $\widetilde{\Lambda }$
for any $\widetilde{\Lambda }\geq \Lambda _{0}$, then $N$ is a \emph{CR}
Kobayashi hyperbolic manifold of order $\Lambda $ for any $\Lambda \in
\lbrack \max \{\sqrt{2}/2,\Lambda _{0}/C\},\infty )$.
\end{corollary}

Using Theorem \ref{theorem6.2}, Proposition \ref{proposition6.3} and Theorem \ref{theorem6.6}, we may construct many
examples of \emph{CR} Kobayashi manifolds of finite order through \emph{CR}
quasi-isometries. For example, we have

\begin{corollary} \label{corollary6.8}
Suppose $(M^{3},H,J,\theta )$ is \emph{CR}
quasi-isometric to $D^{3}(-1)$. Then $M^{3}$ is a \emph{CR} Kobayashi
hyperbolic manifold of order $\Lambda $ for any $\Lambda \in \lbrack \sqrt{2}%
/2,\infty )$.
\end{corollary}

It would also be interesting to investigate the "hyperbolicity problem"
corresponding to the Schwarz type lemmas for transversally holomorphic maps.
Since a transversally holomorphic map induces a map between the leaf-spaces
of two pseudo-Hermitian manifolds, it seems natural to investigate the
distance-decreasing property of this induced map. However, the leafspace of
a pseudo-Hermitian manifold is possibly non-Hausdorff and this may causes
some difficulty. We will study this issue in another paper.

Finally, we would like to mention that the authors in \cite{[CDRY]} have established some Schwarz type lemmas for some generalized holomorphic maps between Kahlerian and pseudo-Hermitian manifolds.

\bigskip

Yuxin Dong \ \ 

School of Mathematical Sciences

Fudan University

Shanghai, 200433, P. R. China

yxdong@fudan.edu.cn \ \ \ 

\bigskip

Yibin Ren

College of Mathematics and Computer Science

Zhejiang Normal University

Jinhua, 321004, Zhejiang, P.R. China

allenryb@outlook.com

\bigskip

Weike Yu

School of Mathematical Sciences

Fudan University

Shanghai, 200433, P. R. China

wkyu2018@outlook.com

\bigskip

\end{document}